%% file: main.tex
\documentclass[11pt, a4paper]{article}

\usepackage[english]{babel}
\usepackage[a4paper]{geometry}
\usepackage{amsmath,amssymb}
\usepackage{amsfonts}
\usepackage{amsthm}
\usepackage{graphicx}
\usepackage[utf8]{inputenc}
\usepackage{mathtools}
\usepackage{booktabs}

\usepackage{color}
\usepackage{xcolor}

\numberwithin{equation}{section}

\theoremstyle{plain}
\newtheorem{theorem}{Theorem}[section]

\newtheorem{corollary}[theorem]{Corollary}

\theoremstyle{definition}

\newtheorem{remark}[theorem]{Remark}
\newtheorem{example}[theorem]{Example}

\newtheorem{algorithm}[theorem]{Algorithm}

\input{def_mathfonts}
\input{def}

\renewcommand{\labelenumi}{\textup{(\roman{enumi})}}

\title{Walsh's Conformal Map onto Lemniscatic Domains for Polynomial Pre-images 
II}
\author{Klaus Schiefermayr\footnotemark[1] \and Olivier 
S\`{e}te\footnotemark[2]}
\date{June 30, 2023}

\begin{document}
\maketitle

\renewcommand{\thefootnote}{\fnsymbol{footnote}}

\footnotetext[1]{University of Applied Sciences Upper Austria, Campus Wels, 
Austria, \\ \texttt{klaus.schiefermayr@fh-wels.at}}

\footnotetext[2]{Institute of Mathematics and Computer Science, University of 
Greifswald, Walther-Rathenau-Stra\ss{}e~47, 17489 Greifswald, Germany.
\texttt{olivier.sete@uni-greifswald.de}}

\renewcommand{\thefootnote}{\arabic{footnote}}

\begin{abstract}
We consider Walsh's conformal map from the exterior of a 
set~$E=\bigcup_{j=1}^\ell E_j$ consisting of $\ell$~compact disjoint components 
onto a lemniscatic domain.  In particular, we are interested in the case 
when~$E$ is a polynomial preimage of $[-1,1]$, i.e., when $E=P^{-1}([-1,1])$, 
where $P$ is an algebraic polynomial of degree~$n$.  Of special interest are 
the exponents and the centers of the lemniscatic domain.  In the first part of 
this series of papers, a very simple formula for the exponents has been 
derived.  In this paper, based on general results of the first part, we give an 
iterative method for computing the centers when $E$ is the union of $\ell$ 
intervals.  Once the centers are known, the corresponding Walsh map can be 
computed numerically.  In addition, if $E$ consists of $\ell=2$ or $\ell=3$ 
components satisfying certain symmetry relations then the centers and the 
corresponding Walsh map are given by explicit formulas.  All our theorems are 
illustrated with analytical or numerical examples.
\end{abstract}



\section{Introduction}

In his 1956 paper~\cite{Walsh1956}, Walsh obtained a canonical generalization
of the Riemann mapping theorem from simply connected domains to multiply
connected domains.  In his construction, an $\ell$-connected domain in 
$\widehat{\C} \coloneq \C \cup \{ \infty \}$ is mapped
onto the exterior of a generalized lemniscate, as indicated in the following
theorem.


\begin{theorem} \label{thm:walsh_map}
Let $E_1, \ldots, E_\ell \subseteq \C$ be disjoint, simply connected, infinite compact sets and let
\begin{equation}
E = \bigcup_{j=1}^\ell E_j,
\end{equation}
that is, $\comp{E} = \widehat{\C} \setminus E$ is an $\ell$-connected domain. Then there exists a unique compact set of the form
\begin{equation} \label{eqn:lemniscatic_domain}
L \coloneq \{ w \in \C : \abs{U(w)} \leq \capacity(E) \}, \quad
U(w) \coloneq \prod_{j=1}^\ell (w-a_j)^{m_j},
\end{equation}
where $a_1, \ldots, a_\ell \in \C$ are distinct and $m_1, \ldots, m_\ell > 0$ are real numbers
with $\sum_{j=1}^\ell m_j = 1$, and a unique
conformal map
\begin{equation}
\Phi : \comp{E} \to \comp{L}
\quad \text{with }
\Phi(z) = z + \cO(1/z) \text{ at } \infty.
\end{equation}
If $E$ is bounded by Jordan curves, then $\Phi$ extends to a homeomorphism from 
$\overline{\comp{E}}$ to $\overline{\comp{L}}$.
\end{theorem}


The compact set $L$ in~\eqref{eqn:lemniscatic_domain} consists of $\ell$ 
disjoint compact components $L_1, \ldots, L_\ell$, with $a_j \in L_j$ for $j = 
1, \ldots, \ell$.
The components $L_1, \ldots, L_\ell$ are labeled such that a Jordan curve 
surrounding $E_j$ is mapped by $\Phi$ onto a Jordan curve surrounding $L_j$.
The \emph{centers} $a_1, \ldots, a_\ell$ and the \emph{exponents} $m_1, \ldots, 
m_\ell$ in Theorem~\ref{thm:walsh_map} are uniquely determined.
The domain $\comp{L}$ is the exterior of the generalized lemniscate 
$\{ w \in \C : \abs{U(w)} = \capacity(E) \}$ and is usually called a 
\emph{lemniscatic domain}; see~\cite[p.~106]{Grunsky1978}.
Here, $\capacity(E)$ denotes the logarithmic capacity of $E$.

Walsh's conformal map onto a lemniscatic domain is a canonical generalization 
of the Riemann map.
Indeed, if $\ell = 1$ in Theorem~\ref{thm:walsh_map}, i.e., if $E$ is simply 
connected, the exterior Riemann map $\cR_E : \comp{E} \to 
\comp{\overline{\bD}}$, uniquely determined by the normalization
$\cR_E(z) = d_1 z + d_0 + \cO(1/z)$ at infinity with $d_1 > 0$, satisfies 
$\cR_E(z) = d_1 \Phi(z) + d_0$, which follows from~\cite[Thm.~4]{Walsh1956}; 
see also~\cite[Rem.~1.2]{SchiefermayrSete2022}.
Thus, $\cR_E$ and $\Phi$ are related by a simple linear transformation, and
$L = \{ w \in \C : \abs{w - a_1} \leq \capacity(E) \}$ is a disk, 
where $a_1 = - d_0/d_1$ and $\capacity(E) = 1/d_1$.

The reason why we are interested in Walsh's conformal map is that the 
lemniscatic domain $\comp{L}$ in~\eqref{eqn:lemniscatic_domain} has a very 
simple form and is in particular (the exterior of) a classical lemniscate if 
the exponents $m_1, \ldots, m_\ell$ in~\eqref{eqn:lemniscatic_domain} are 
rational.
In addition, as in the case of the Riemann map where the Green's function for 
the complement of the unit disk is simply $\log \abs{w}$, also the Green's 
function for the complement of $L$ in~\eqref{eqn:lemniscatic_domain} has the 
simple form
\begin{equation}
g_L(w) = \log \abs{U(w)} - \log(\capacity(E))
= \sum_{j=1}^\ell m_j \log \abs{w - a_j} - \log(\capacity(E))
\end{equation}
and $g_E(z) = g_L(\Phi(z))$ holds with $\Phi$ from Theorem~\ref{thm:walsh_map}.
Moreover, Walsh's conformal map allows the construction of Faber--Walsh 
polynomials on sets $E$ with several components as in 
Theorem~\ref{thm:walsh_map}, generalizing the well-known Faber polynomials and 
the classical Chebyshev polynomials of the first kind; see Walsh's original 
article~\cite{Walsh1958}, the book of Suetin~\cite{Suetin1998}, and the 
article~\cite{SeteLiesen2017}.

After Walsh's seminal paper~\cite{Walsh1956}, further existence proofs 
of Theorem~\ref{thm:walsh_map} were published by Grunsky~\cite{Grunsky1957a, 
Grunsky1957b,Grunsky1978}, Jenkins~\cite{Jenkins1958}, and 
Landau~\cite{Landau1961}.
The latter contains an iteration for computing Walsh's map,
but it requires knowledge of the harmonic measure of the boundary.  
None of these papers contain any explicit example, which
might be the reason that Walsh's map has not been widely used so far.
However, in Walsh’s Selected Papers \cite[pp.~374--377]{Walsh2000}, Gaier 
recognizes Walsh's conformal map onto lemniscatic domains as one of Walsh's 
major contributions.
The first explicit examples of Walsh's map were derived 
in~\cite{SeteLiesen2016} and applied in~\cite{SeteLiesen2017} for polynomial 
approximation on disconnected compact sets.
In~\cite{NasserLiesenSete2016}, Nasser, Liesen and the second author obtained a 
numerical method for computing Walsh's conformal map for sets $E$ bounded by 
smooth Jordan curves.  The method relies on solving a boundary integral 
equation (BIE) with the generalized Neumann kernel.  This numerical algorithm also 
yields a method for the numerical computation of the logarithmic capacity of 
compact sets~\cite{LiesenSeteNasser2017}.
In~\cite{SchiefermayrSete2022}, we obtained a characterization of the exponents 
$m_1, \ldots, m_\ell$ in terms of Green's function 
and derived explicit formulas and examples of conformal maps onto lemniscatic 
domains for polynomial pre-images $E$.

One main objective is the actual \emph{computation} of the lemniscatic domain
$L$, that is, the computation of the exponents~$m_j$ and the centers~$a_j$,
and of the conformal map~$\Phi$.
In this paper, we consider this question for sets~$E$ that are polynomial
pre-images of $\cc{-1, 1}$, i.e., $E = P^{-1}(\cc{-1, 1})$, where $P$ is an
algebraic polynomial of degree $n$.  The set $P^{-1}(\cc{-1, 1})$ consists
of $\ell$ components, $1 \leq \ell \leq n$, where each component consists 
of a certain number of analytic Jordan arcs~\cite{Schiefermayr2012}.
In particular, the components $E_j$ of $P^{-1}(\cc{-1, 1})$ are not bounded
by Jordan curves and thus the BIE method for computing $\Phi$ and $L$
from~\cite{NasserLiesenSete2016} does not apply here.

This paper is the second of a series of papers.
In the first part~\cite{SchiefermayrSete2022},
we characterized the exponents as $m_j = n_j/n$, where $n_j$ is the
number of zeros of the polynomial $P$
in the component $E_j$ and $n$ is the degree of the polynomial $P$.
Moreover, we derived some general characterizations for the centers $a_1, 
\ldots, a_\ell$ and considered the cases $\ell = 1$ and $\ell = 2$ in 
more detail.

Building on these results, the main contributions in this paper are the following.

(1) In Section~\ref{sect:two_components}, we consider sets $E$ consisting
of $\ell = 2$ real intervals, or, more 
generally, of two components that are symmetric with respect to the real
line.
For these sets, we have explicit formulas for the centers $a_1, a_2$ (Theorem~\ref{thm:aj_for_two_components}).
If, in addition, the set~$E$ is symmetric with respect to the origin then
the corresponding Walsh map $\Phi$ can be given explicitly; see
Theorem~\ref{thm:double_symmetry} and the following two examples.

(2) In Section~\ref{sect:three_components}, we consider sets $E$ consisting of
$\ell = 3$ real intervals.  In this case, the centers $a_1, a_2, a_3$
are the solution of a nonlinear system of three equations,
which can be solved numerically.
This is illustrated in Example~\ref{ex:three_intervals} considering a
two-parameter family of three intervals.
If, in addition, the set~$E$ is symmetric with respect to the origin then
$a_1, a_2, a_3$ can be given by an explicit formula.
In either case, the conformal map $\Phi$ can be computed numerically by 
solving a polynomial equation and using the mapping properties of $\Phi$ in Theorem~\ref{thm:mapping_properties_Phi}.

(3) In Section~\ref{sect:algorithm},
we propose an iterative method for the computation of $a_1, \ldots, 
a_\ell$ for sets $E$ consisting of an 
arbitrary number $\ell \geq 1$ of real intervals, or, more generally, of $\ell$ 
components symmetric with respect to the real line.
The key idea is that $a_1, \ldots, a_\ell$ are the zeros of a certain 
polynomial with prescribed critical values but unknown critical points, with an 
additional constraint on $a_1, \ldots, a_\ell$.
Once the centers $a_j$ of $L$ are computed, the conformal map $\Phi$ can be 
evaluated numerically.
The iterative method for the computation of $a_1, \ldots, a_\ell$ works for an 
arbitrary number of intervals.
In all our numerical examples, the method converges in very few iteration 
steps (at most $7$ steps for up to $20$ intervals) and returns highly 
accurate approximations of $a_1, \ldots, a_\ell$.
In all examples where $a_1, \ldots, a_\ell$ are known explicitly, the computed 
values have an error of the order of the machine precision.
In those examples, where the exact values of $a_1, \ldots, a_\ell$ are not 
known explicitly, the values obtained with 
our new method agree to order $10^{-14}$ with those obtained by a modification 
of the BIE method from~\cite{NasserLiesenSete2016}.  This suggests that both 
methods are very accurate.

In Section~\ref{sect:background}, we first collect results on Walsh's conformal
map~$\Phi$ and the lemniscatic domain for general compact 
sets~$E$ with certain symmetries.
Second, we recall known important facts on sets~$E$ which are polynomial pre-images.
In Appendix~\ref{sect:appendix}, a rather general result concerning symmetric
sets is proven, which is needed for the proofs in
Sections~\ref{sect:two_components} and~\ref{sect:three_components}.

\section{General Compact Sets and Polynomial Pre-Images}
\label{sect:background}

\subsection{Results for General Compact Sets}

For a set $K \subseteq \widehat{\C}$, we denote as usual
$K^* \coloneq \{ \conj{z} : z \in K \}$ and $-K \coloneq \{ -z : z \in K \}$.
If $K \subseteq \C$ is compact then the Green's function (with pole at 
infinity) of $\comp{K}$ is denoted by $g_K$.
If $E = E_1 \cup \ldots \cup E_\ell$ and $E_j^* = E_j$ for all components of
$E$ then we label the components ``from left to right'':
By~\cite[Lem.~A.2]{SchiefermayrSete2022}, each 
$E_j \cap \R$ is a point or an interval, and we label $E_1, \ldots, E_\ell$
such that $x \in E_j \cap \R$ and $y \in E_{j+1} \cap \R$ implies $x < y$ for 
all $j = 1, \ldots, \ell-1$.
As a first result, let us collect some mapping properties of $\Phi$ if the set $E$ has certain symmetries.  

\begin{theorem} \label{thm:mapping_properties_Phi}
Let the notation be as in Theorem~\ref{thm:walsh_map}.
\begin{enumerate}
\item If $E^* = E$, then $\Phi(z) = \conj{\Phi(\conj{z})}$, $\Phi(\R \setminus 
E) = \R \setminus L$, and $\Phi$ maps the upper (lower) half-plane without $E$ 
bijectively onto the upper (lower) half-plane without $L$.

\item If $E = -E^*$, i.e., $E$ is symmetric with respect to the imaginary axis,
then  $\Phi(z) = -\conj{\Phi(-\conj{z})}$,
$\Phi( \ii \R \setminus E) = \ii \R \setminus L$, and
$\Phi$ maps the left (right) half-plane without $E$ bijectively onto the left 
(right) half-plane without $L$.

\item If $E^* = E$ and $E = -E$, then $\Phi(\R^\pm \setminus E) = \R^\pm 
\setminus L$, $\Phi( \ii \R^\pm \setminus E) = \ii \R^\pm \setminus L$, and
$\Phi(z)$ is in the same quadrant as $z$.

\item Assume that $E_j^* = E_j$ for $j = 1, \ldots, \ell$.
Let $\R \setminus E = I_0 \cup \ldots \cup I_\ell$ with open intervals 
$I_j$ ordered from left to right, and let similarly $\R \setminus L = J_0 \cup 
\ldots \cup J_\ell$ with open intervals $J_j$ ordered from left to right. Then 
$\Phi$ maps $I_j$ onto $J_j$, that is $\Phi(I_j) = J_j$ for $j = 0, \ldots, 
\ell$, and $\Phi$ is strictly increasing on $\R \setminus E$, and 
in particular on each interval $I_j$.

\item \label{it:known_properties_crit_pts}
Assume that $E_j^* = E_j$ for $j = 1, \ldots, \ell$.
Then the Green's function $g_E$ has $\ell-1$ critical points $z_1, \ldots, 
z_{\ell-1} \in \C \setminus E$.  These are simple and satisfy $z_j \in I_j$ 
for $j = 1, \ldots, \ell-1$.  Moreover, $\Phi(z_j) = w_j \in J_j$ are the 
critical points of $g_L$.  In particular, if $z \in I_j$ with 
$z < z_j$ (resp.\ $z > z_j$) then $w = \Phi(z) \in J_j$ with $w < w_j$ 
(resp.\ $w > w_j$), and
\begin{equation}
a_1 < w_1 < a_2 < w_2 < \ldots < w_{\ell-1} < a_\ell.
\end{equation}
\end{enumerate}
\end{theorem}

\begin{proof}
(i)
Since $E^* = E$, we have $\Phi(z) = \conj{\Phi(\conj{z})}$ for $z \in \C 
\setminus E$ by~\cite[Lem.~2.2]{SeteLiesen2016}, hence $\Phi(z) \in \R$ if and 
only if $z \in \R$.  Then the normalization $\Phi(z) = z + \cO(1/z)$ at 
infinity implies that
$\im(\Phi(z)) > 0$ for $z \in \C \setminus E$ with $\im(z) > 0$, and
$\im(\Phi(z)) < 0$ for $z \in \C \setminus E$ with $\im(z) < 0$.
(ii) Since $E = -E^*$, we have $\Phi(z) = -\conj{\Phi(-\conj{z})}$ for
$z \in \C \setminus E$ by~\cite[Lem.~2.2]{SeteLiesen2016}, and (ii) now follows 
similarly to (i).
(iii) follows from (i) and (ii).
The assertions (iv) and (v) follow from~\cite[Thm.~2.8]{SchiefermayrSete2022} 
and its proof.
\end{proof}

In the next theorem, we consider sets $E$ that are symmetric with respect to
the origin and investigate the effect on the exponents $m_j$ and the centers
$a_j$ of the lemniscatic domain.

\begin{theorem} \label{thm:aj_symm_origin}
Let $E = \cup_{j=1}^\ell E_j$ be as in Theorem~\ref{thm:walsh_map} with centers $a_1, \ldots, a_\ell$ and exponents $m_1, \ldots, m_\ell$ of $L$.
If $E = -E$ and if $j_1, j_2$ 
are such that $E_{j_2} = - E_{j_1}$, then $m_{j_2} = m_{j_1}$ and $a_{j_2} = 
-a_{j_1}$.  In particular, if $E_j = - E_j$, then $a_j = 0$.
Moreover, the set of critical points of $g_E$ is symmetric with respect to zero.
\end{theorem}

\begin{proof}
If $E = -E$ then $\Phi(-z) = -\Phi(z)$ by~\cite[Lem.~2.2]{SeteLiesen2016} 
or~\cite[Lem.~2.6]{SchiefermayrSete2022}, and $g_E(-z) = g_E(z)$ since 
$g_E(z)$ and $g_E(-z)$ both satisfy the properties of the Green's function,
hence, in particular, $(\partial_z g_E)(-z) = - (\partial_z g_E)(z)$.
This shows that the set of critical points of $g_E$ is symmetric with respect
to zero.
Next, let $\gamma$ be a smooth closed curve in $\C \setminus E$ with
$\wind(\gamma; z) = \delta_{j_1,k}$ for $z \in E_k$,
then $-\gamma$ surrounds the component $-E_{j_1} = E_{j_2}$.
We use~\cite[Thm.~2.3]{SchiefermayrSete2022} and the substitution $z = -u$ to 
obtain
\begin{align*}
m_{j_2}
= \frac{1}{2 \pi \ii} \int_{-\gamma} 2 (\partial_z g_E)(z) \, \dd z
= -\frac{1}{2 \pi \ii} \int_\gamma 2 (\partial_z g_E)(-u) \, \dd u
= \frac{1}{2 \pi \ii} \int_\gamma 2 (\partial_z g_E)(u) \, \dd u
= m_{j_1}.
\end{align*}
Similarly,
\begin{align*}
m_{j_2} a_{j_2}
&= \frac{1}{2 \pi \ii} \int_{-\gamma} \Phi(z) 2 (\partial_z g_E)(z) \, \dd z
= -\frac{1}{2 \pi \ii} \int_\gamma \Phi(-u) 2 (\partial_z g_E)(-u) \, \dd u \\
&= - \frac{1}{2 \pi \ii} \int_\gamma \Phi(u) 2 (\partial_z g_E)(u) \, \dd u
= - m_{j_1} a_{j_1},
\end{align*}
hence $a_{j_2} = - a_{j_1}$.
\end{proof}

\subsection{Results for Polynomial Pre-Images of an Interval}
\label{sect:poly_preimages_of_intervals}


We consider polynomial pre-images of $\cc{-1, 1}$, that is,
\begin{equation} \label{eqn:poly_preimage}
E \coloneq P^{-1}(\cc{-1, 1}) = \bigcup_{j=1}^\ell E_j,
\end{equation}
where $P$ is a polynomial of degree $n \geq 1$ with complex coefficients of the form
\begin{equation} \label{eqn:P}
P(z) = \sum_{j=0}^n p_j z^j
= p_n z^n + p_{n-1} z^{n-1} + \ldots + p_0
\quad \text{with } p_n \neq 0.
\end{equation}
Each component $E_j$ consists of a certain number of analytic Jordan arcs, 
see~\cite{Schiefermayr2012} for details, and $\comp{E}$ is connected; 
see~\cite[Thm.~A.4]{SchiefermayrSete2022} 
or~\cite[Lem.~1\,(viii)]{Schiefermayr2012}.  All zeros of $P$ 
are in $E$ and each $E_j$ contains at least one zero of $P$, compare 
\cite[Thm.~3.1]{SchiefermayrSete2022}.
By~\cite[Proof of Thm.~5.2.5]{Ransford1995}, the Green's function 
of $\comp{E}$ is given by
\begin{equation} \label{eqn:gE}
g_E(z)
= \frac{1}{n} \log \bigl\lvert P(z) + \sqrt{P(z)^2-1} \bigr\rvert, \quad z \in \C \setminus E,
\end{equation}
where $\sqrt{P(z)^2 - 1}$ has a branch cut along $E$ and behaves as $P(z)$ 
at $\infty$; compare also the beginning of~\cite[Sect.~3]{SchiefermayrSete2022}.
Note that the critical points of $g_E$ are the critical points of~$P$ in
$\C \setminus E$.
Again by~\cite[Thm.~5.2.5]{Ransford1995}, the logarithmic capacity of $E$ is
\begin{equation} \label{eqn:capacity}
\capacity(E) = \frac{1}{\sqrt[n]{2 \abs{p_n}}}.
\end{equation}


In the next theorem, we summarize results from~\cite{SchiefermayrSete2022}
that are needed in the sequel.

\begin{theorem} \label{thm:known_properties}
Let $E$ be a polynomial pre-image as in~\eqref{eqn:poly_preimage}.
\begin{enumerate}
\item The set $E$ has $\ell$ connected components
if and only if $P$ has exactly $\ell-1$ critical points $z_1, \ldots, 
z_{\ell-1}$ (counting multiplicities) for which $P(z_k) \notin \cc{-1, 1}$ for 
$k = 1, \ldots, \ell-1$.
\item The exponents $m_j$ in Theorem~\ref{thm:walsh_map} are given by
\begin{equation}
m_j = \frac{n_j}{n}, \quad j = 1, \ldots, \ell,
\end{equation}
where $n_j \geq 1$ is the number of zeros of $P$ in $E_j$.
\item \label{it:relation_between_maps}
For $z \in \comp{E}$, we have
\begin{equation} \label{eqn:relation_between_maps}
Q(\Phi(z)) = P(z) + \sqrt{P(z)^2 - 1}
\end{equation}
with the polynomial
\begin{equation} \label{eqn:Q}
Q(w) \coloneq \frac{\ee^{\ii \arg(p_n)}}{\capacity(E)^n} U(w)^n
= 2 p_n \prod_{j=1}^\ell (w-a_j)^{n_j}.
\end{equation}
Moreover,
\begin{equation}
L = \{ w \in \C : \abs{Q(w)} \leq 1 \},
\end{equation}
compare~\eqref{eqn:lemniscatic_domain}.
\item A point $z \in \C \setminus E$ is a critical point of $P$ if and 
only if $w = \Phi(z)$ is a critical point of $Q$ in $\C \setminus L$, and
\begin{equation} \label{eqn:relation_critvals}
Q(w) = P(z) + \sqrt{P(z)^2 - 1}.
\end{equation}
\item The polynomial $Q$ has $\ell-1$ critical points in $\C \setminus L$
and these are the zeros of
\begin{equation} \label{eqn:eqn_for_crit_pts_of_Q}
\sum_{k=1}^\ell n_k \prod_{\substack{j=1 \\ j \neq k}}^\ell (w - a_j).
\end{equation}
\end{enumerate}
\end{theorem}

\begin{proof}
The exterior Riemann map of $\cc{-1, 1}$ is $\Ri : \comp{\cc{-1, 1}} \to \comp{\overline{\bD}}$, $\Ri(z) = z + \sqrt{z^2 - 1}$, 
where $\sqrt{z^2 - 1}$ has a branch cut along $\cc{-1, 1}$ and behaves as $z$ 
at $\infty$.
Then
(i) follows from~\cite[Thm.~3.1]{SchiefermayrSete2022},
(ii) from~\cite[Thm.~3.2]{SchiefermayrSete2022},
(iii) from~\cite[Thm.~3.3]{SchiefermayrSete2022},
(iv) and (v) from~\cite[Lem.~3.5]{SchiefermayrSete2022}.
\end{proof}

Since the exponents $m_j$ are determined in 
Theorem~\ref{thm:known_properties}\,(ii),
let us turn our attention to the determination of the centers $a_j$.
If all components $E_j$ of $E$ are symmetric with respect to the real line,
we know that $a_1, \ldots, a_\ell$ are the solution of a certain nonlinear
system of equations.

\begin{theorem} \label{thm:ell_components_real_symmetry}
Let $P$ be a polynomial of degree $n \geq 1$ such that
$E = P^{-1}(\cc{-1, 1}) = E_1 \cup \ldots \cup E_\ell$ with
$E_j^* = E_j$ for $j = 1, \ldots, \ell$.
Then $a_1, \ldots, a_\ell$ are real with $a_1 < \ldots < a_\ell$, the 
polynomial $Q(w)$ in~\eqref{eqn:Q} has exactly $\ell-1$ critical points $w_1, 
\ldots, w_{\ell-1} \in \C \setminus L$,
and these are real, simple with $a_1 < w_1 < a_2 < w_2 < a_3 < \ldots < a_\ell$.
Similarly, $P$ has $\ell-1$ critical points $z_1, \ldots, z_{\ell-1}$
in $\C \setminus E$, these are real, simple and satisfy
\begin{equation}
\max(E_j \cap \R) < z_j < \min(E_{j+1} \cap \R), \quad j = 1, \ldots, \ell-1.
\end{equation}
Moreover, $a_1, \ldots, a_\ell$ satisfy the system of equations
\begin{align}
Q(w_j) &= P(z_j) + \sqrt{P(z_j)^2 - 1}, \qquad j = 1, \ldots, \ell-1, 
\label{eqn:ell_relation_crit_pts} \\
\sum_{j=1}^\ell n_j a_j &= - \frac{p_{n-1}}{p_n}. \label{eqn:ell_sum_nj_aj}
\end{align}
\end{theorem}

\begin{proof}
The first part follows immediately from 
Theorem~\ref{thm:mapping_properties_Phi}\,(v).  In particular, $\Phi(z_j) = 
w_j$ for $j = 1, \ldots, \ell-1$.  Then, \eqref{eqn:ell_relation_crit_pts} 
follows from Theorem~\ref{thm:known_properties}\,(iv), 
and~\eqref{eqn:ell_sum_nj_aj} follows 
from~\cite[Thm.~3.3]{SchiefermayrSete2022}.
\end{proof}

Equations~\eqref{eqn:ell_relation_crit_pts}--\eqref{eqn:ell_sum_nj_aj}
are a nonlinear system of equations for $a_1, \ldots, a_\ell$,
provided that the critical points $w_1, \ldots, w_{\ell-1}$ of $Q$ outside $L$
can be expressed explicitly in terms of $a_1, \ldots, a_\ell$;
see~\eqref{eqn:eqn_for_crit_pts_of_Q}.
This is possible for $\ell = 2$ and $\ell = 3$, and we consider
these cases in Sections~\ref{sect:two_components} 
and~\ref{sect:three_components}.
For larger $\ell$, this is no longer practical. 
In Section~\ref{sect:algorithm}, we develop a numerical method to compute
$a_1, \ldots, a_\ell$ for arbitrary $\ell \geq 2$.

Before we start this investigation, let us conclude this section with a theorem
which gives a necessary and sufficient condition for a
union of real intervals to be a polynomial pre-image of $\cc{-1, 1}$,
and which follows from~\cite[Lem.~2.1]{PeherstorferSchiefermayr1999}; see
also~\cite[Lem.~1]{Schiefermayr2007}.

\begin{theorem} \label{thm:E_ell_intervals}
The set
\begin{equation} \label{eqn:E_ell_intervals}
E \coloneq \cc{c_1, c_2} \cup \cc{c_3, c_4} \cup \cc{c_5, c_6} \cup
\ldots \cup \cc{c_{2n-1}, c_{2n}}
\end{equation}
with
\begin{equation}
c_1 < c_2 \leq c_3 < c_4 \leq c_5 < c_6 \leq \ldots \leq c_{2n-1} < c_{2n}
\end{equation}
is the pre-image of $\cc{-1, 1}$ under a polynomial $P$ of degree $n$, that is, 
$E = P^{-1}(\cc{-1, 1})$, if and only if
$c_1, \ldots, c_{2n}$ satisfy the following system of equations:
\begin{equation} \label{eqn:system_bk}
c_1^k - (c_2^k + c_3^k) + (c_4^k + c_5^k) - (c_6^k + c_7^k) \pm \ldots
+ (-1)^{n-1} (c_{2n-2}^k + c_{2n-1}^k) + (-1)^n c_{2n}^k = 0
\end{equation}
for $k = 1, 2, \ldots, n-1$.
\end{theorem}

\begin{remark} \label{rem:E_ell_intervals}
\begin{enumerate}
\item The set $E$ in Theorem~\ref{thm:E_ell_intervals} is the union of
$\ell$ \emph{disjoint} intervals, where $\ell \in \{ 1, 2, \ldots, n \}$,
that is, $E$ has $\ell-1$ \emph{gaps}.

\item \label{it:E_ell_intervals_P}
With the help of $c_1, \ldots, c_{2n}$ satisfying~\eqref{eqn:system_bk},
the corresponding polynomial $P$ of degree $n$ (unique up to sign) can be given 
by
\begin{equation} \label{eqn:P_for_E_ell_intervals}
\begin{aligned}
P(z) &= 1 - 2 \frac{(z-c_1) (z - c_4) (z - c_5) (z - c_8) (z - c_9) 
\ldots}{(c_2 - c_1) (c_2 - c_4) (c_2 - c_5) (c_2 - c_8) (c_2 - c_9) \ldots} \\
&= -1 + 2 \frac{(z-c_2) (z - c_3) (z - c_6) (z - c_7) (z - c_{10}) (z - c_{11}) 
\ldots}{(c_1-c_2) (c_1 - c_3) (c_1 - c_6) (c_1 - c_7) (c_1 - c_{10}) (c_1 - 
c_{11}) \ldots}.
\end{aligned}
\end{equation}
Note that $P(z)$ and $-P(z)$ generate the same pre-image $E$.

\item Figure~\ref{fig:E_ell_intervals} shows the real graph of a
polynomial of degree $n = 7$ whose pre-image consists of the $\ell = 4$
disjoint intervals $E = \cc{c_1, c_4} \cup \cc{c_5, c_6} \cup \cc{c_7, c_{12}} \cup \cc{c_{13}, c_{14}}$, which has been computed with the help
of~\eqref{eqn:system_bk} and~\eqref{eqn:P_for_E_ell_intervals}.
\end{enumerate}
\end{remark}

\begin{figure}[t!]
{\centering
\includegraphics[width=0.98\linewidth]{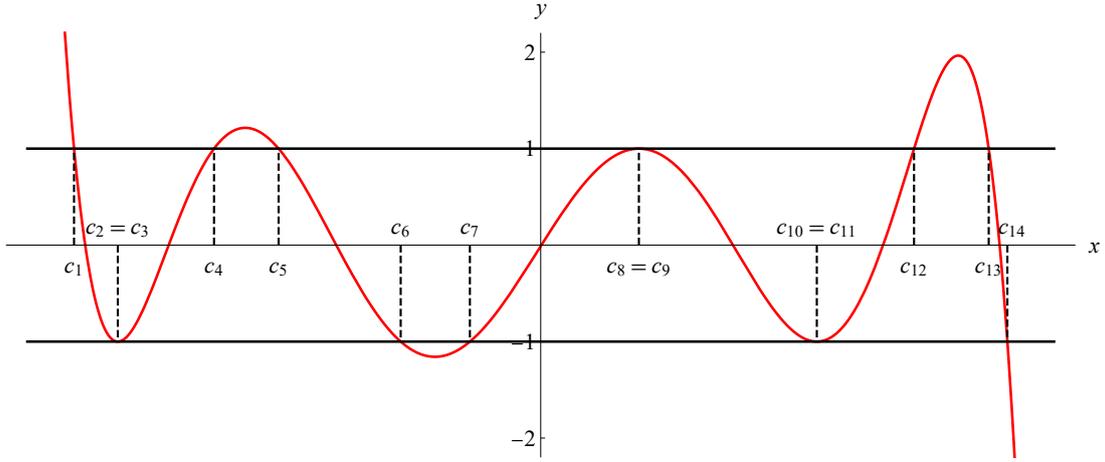}

}
\caption{Real graph of a polynomial $P$ of degree $n = 7$ whose pre-image of 
$\cc{-1, 1}$ consists of $\ell = 4$ intervals; see Example~\ref{ex:deg7_ell4} 
for the explicit formula of $P$.}
\label{fig:E_ell_intervals}
\end{figure}

\section{Sets with Two Components}\label{sect:two_components}

If the polynomial pre-image $E = P^{-1}(\cc{-1, 1})$ has two components
that are
symmetric with respect to the real line, the centers $a_1, a_2$ are given
explicitly in 
Theorem~\ref{thm:aj_for_two_components}, which follows from the more general 
result in~\cite[Cor.~3.10]{SchiefermayrSete2022}.
This covers the important case that $E$ consists of two real intervals,
that is
$E = \cc{b_1, b_2} \cup \cc{b_3, b_4}$
with $b_1 < b_2 < b_3 < b_4$.


\newpage

\begin{theorem}
\label{thm:aj_for_two_components}
Let $P$ be a polynomial of degree $n \geq 1$ as in~\eqref{eqn:P} with either real or 
purely imaginary coefficients such that $E = P^{-1}(\cc{-1, 1}) = E_1 \cup E_2$ 
with $E_1^* = E_1$ and $E_2^* = E_2$.
Let $n_1, n_2$ be the number of zeros of $P$ in $E_1$, $E_2$, respectively, and 
let $z_1$ be the critical point of $P$ in $\C \setminus E$. Then the points 
$a_1, a_2$ are real with $a_1 < a_2$ and are given by
\begin{equation} \label{eqn:a1a2_for_two_intervals}
\begin{aligned}
a_1 &= - \frac{p_{n-1}}{n p_n} - \bigg( \left( \frac{n_2}{n_1} \right)^{n_1} 
\frac{(-1)^{n_2}}{2 p_n} \Big( P(z_1) + \sqrt{P(z_1)^2 - 1} \Big) 
\bigg)^{1/n},\\
a_2 &= - \frac{p_{n-1}}{n p_n} + \bigg( \left( \frac{n_1}{n_2} \right)^{n_2} 
\frac{(-1)^{n_2}}{2 p_n} \Big( P(z_1) + \sqrt{P(z_1)^2 - 1} \Big) 
\bigg)^{1/n},
\end{aligned}
\end{equation}
with the positive real $n$-th root in~\eqref{eqn:a1a2_for_two_intervals}.
\end{theorem}

\begin{figure}[t]
{\centering
\includegraphics[width=0.48\linewidth]{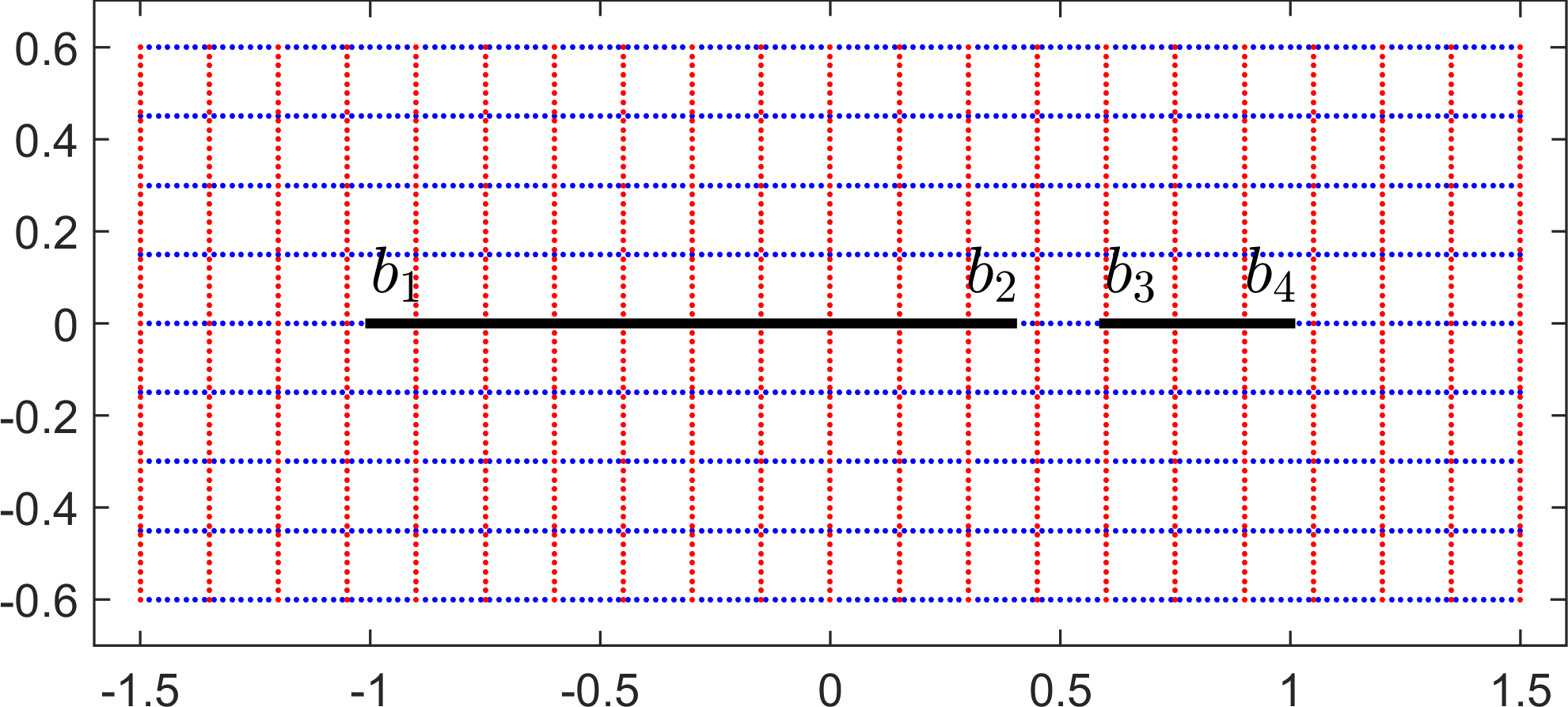}
\includegraphics[width=0.48\linewidth]{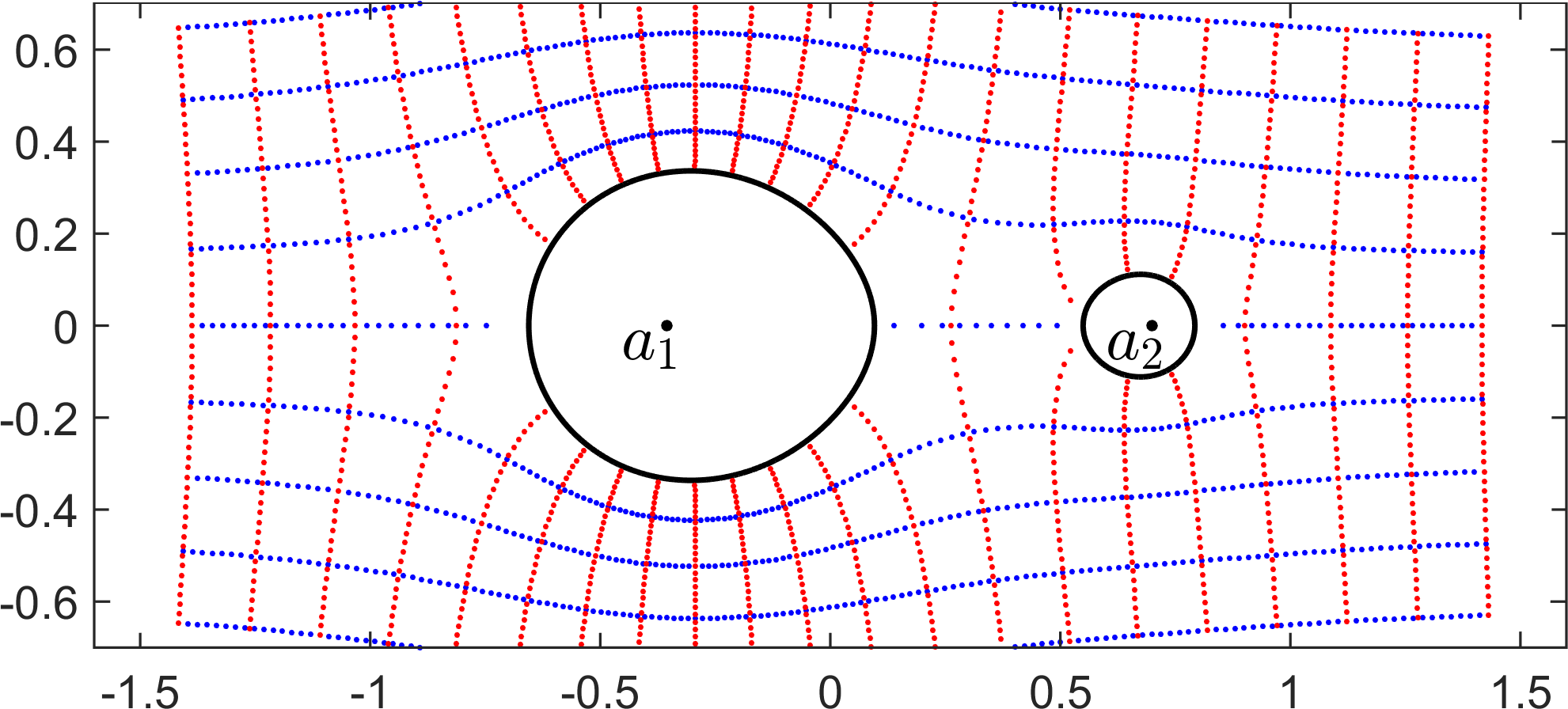}

\includegraphics[width=0.48\linewidth]{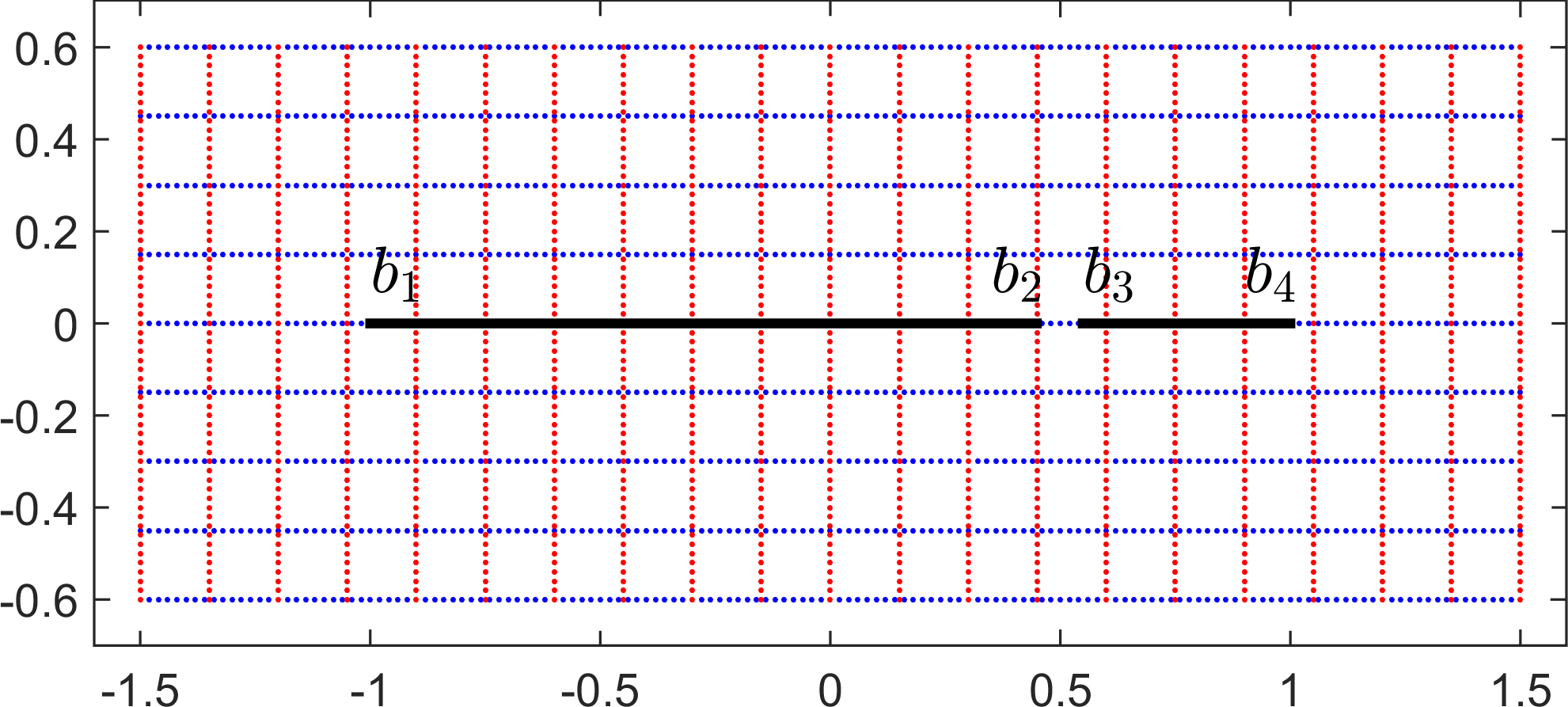}
\includegraphics[width=0.48\linewidth]{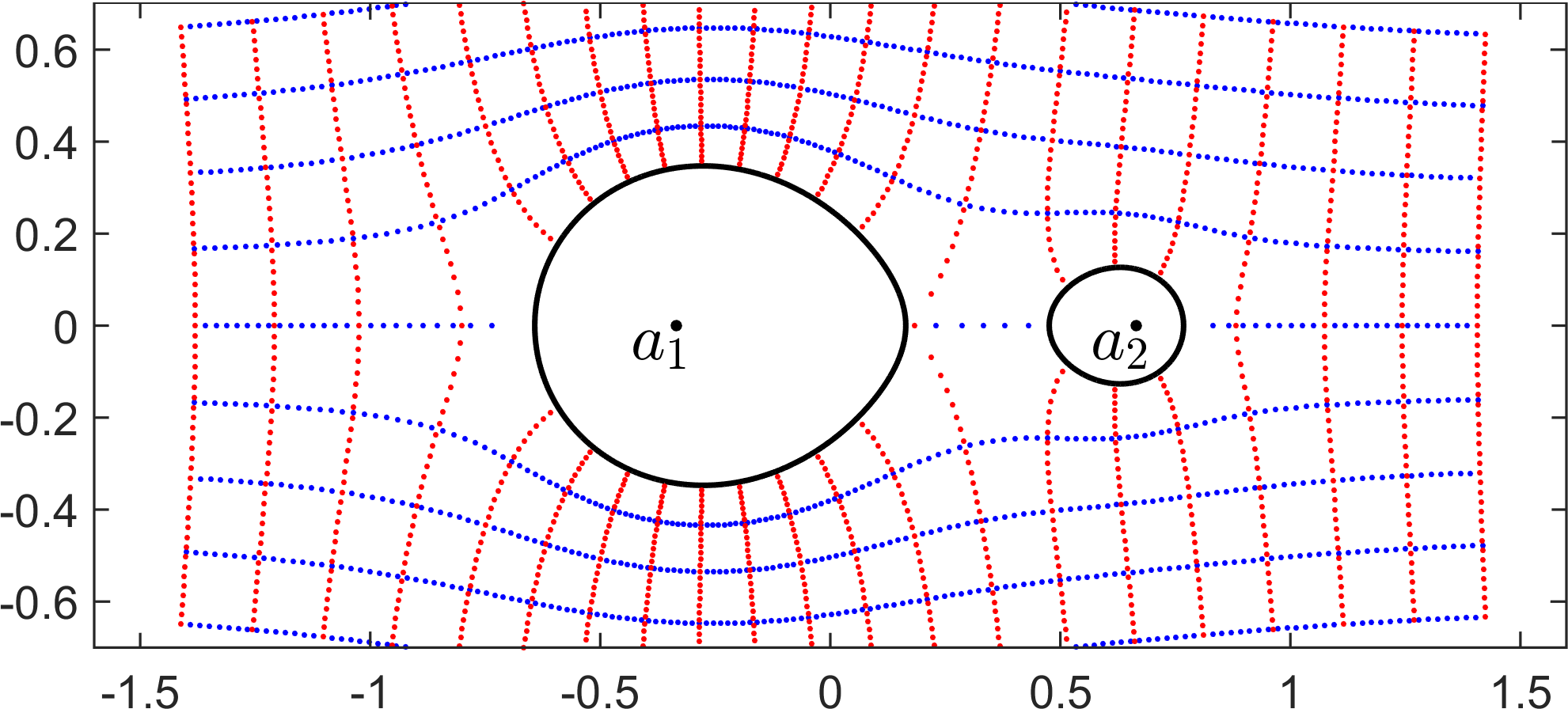}

\includegraphics[width=0.48\linewidth]{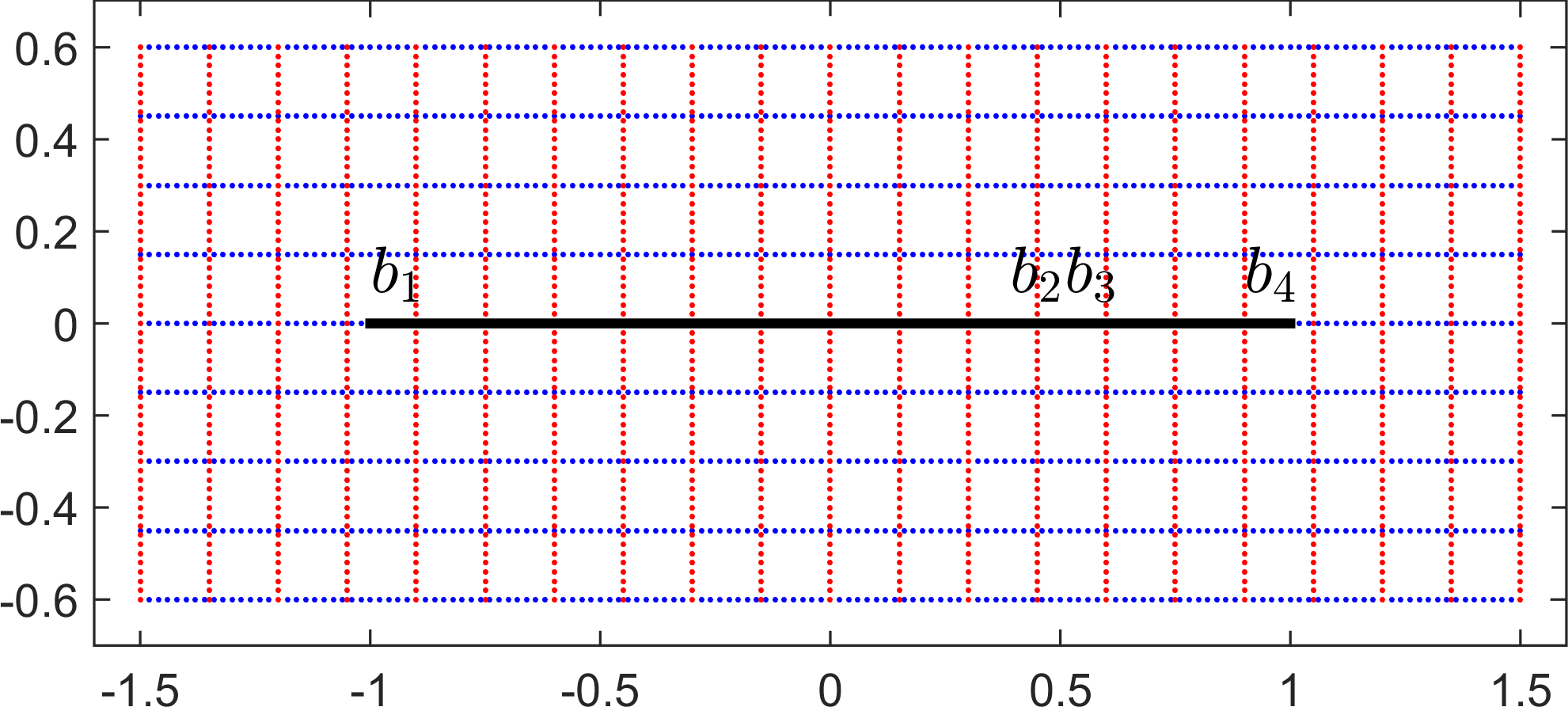}
\includegraphics[width=0.48\linewidth]{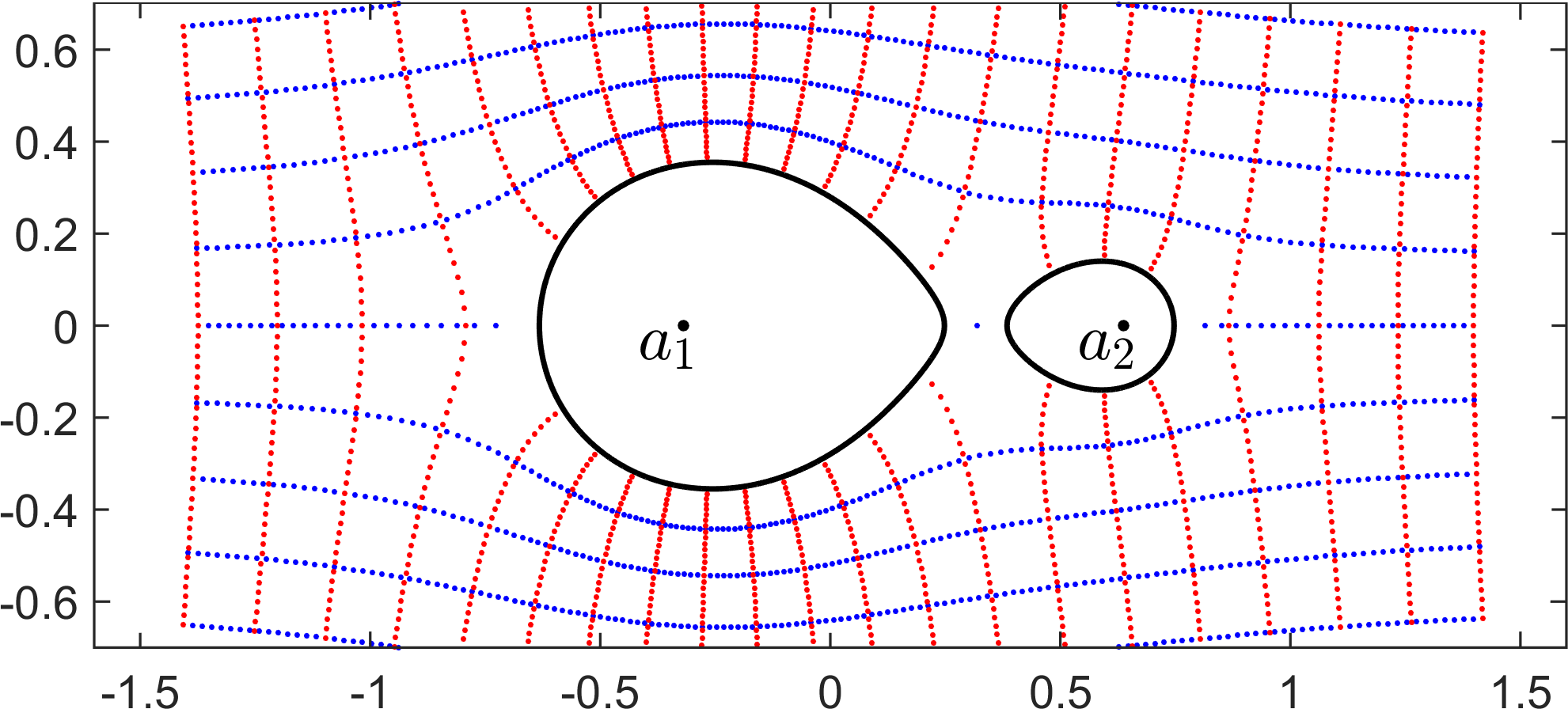}

}
\caption{Illustration of the Walsh map $\Phi$ for $E = \cc{-1, b_2} \cup 
\cc{b_3, 1}$ with $b_2 = \frac{1}{2} (1 - \alpha^2) - \alpha$ and $b_3 = 
\frac{1}{2} (1 - \alpha^2) + \alpha$ for $\alpha = 0.1, 0.05, 0.01$ (from top 
to bottom); see Example~\ref{ex:P3_two_intervals}.
Left: Original domain with intervals (black) and a grid.
Right: $\partial L$ (black) and image of the grid under $\Phi$.
}
\label{fig:P3_two_intervals}
\end{figure}


In the following example for a one-parameter family of two intervals,
we apply the formulas of Theorem~\ref{thm:aj_for_two_components}
in order to obtain the lemniscatic domain and compute the corresponding
conformal map~$\Phi$.

\begin{example} \label{ex:P3_two_intervals}
For $0 < \alpha < 1$, consider the polynomial
\begin{equation*}
P(z) = 1 + \frac{(z-1) (2 z + 1 + \alpha^2)^2}{(1-\alpha^2)^2}
= \frac{4 z^3 + 4 \alpha^2 z^2 + (\alpha^4 - 2 \alpha^2 + 3) z - 4 \alpha^2}{(1-\alpha^2)^2}
\end{equation*}
of degree $n = 3$. It is easy to see that
$E = P^{-1}(\cc{-1, 1}) = [-1, b_2] \cup [b_3, 1]$,
where
$b_2 = \frac{1}{2} (1 - \alpha^2) - \alpha$,
$b_3 = \frac{1}{2} (1 - \alpha^2) + \alpha$.
We have $n_1 = 2$, $n_2 = 1$,
$p_3 = 4/(1-\alpha^2)^2$, $p_2 = 4 \alpha^2/(1-\alpha^2)^2$,
and the critical point of $P$ outside 
$E$ is $z_1 = \frac{1}{6} (3 - \alpha^2)$.  Then, using the correct 
branch of the square root as indicated after~\eqref{eqn:gE}, we obtain
\begin{align*}
P(z_1) + \sqrt{P(z_1)^2 - 1}
&= - \frac{2 \alpha^6 - 9 \alpha^4 + 108 \alpha^2 + 27 + 2 \alpha (9 - 
\alpha^2) (3 + \alpha^2)^{3/2}}{27 (1 - \alpha^2)^2}.
\end{align*}
By~\eqref{eqn:a1a2_for_two_intervals}, the centers are
\begin{align*}
a_1 &= - \frac{\alpha^2}{3} - \frac{1}{6 \sqrt[3]{4}} \Big(
2 \alpha^6 - 9 \alpha^4 + 108 \alpha^2 + 27 + 2 \alpha (9 - 
\alpha^2) (3 + \alpha^2)^{3/2} \Big)^{1/3}, \\
a_2 &= - \frac{\alpha^2}{3} + \frac{1}{3 \sqrt[3]{4}} \Big(
2 \alpha^6 - 9 \alpha^4 + 108 \alpha^2 + 27 + 2 \alpha (9 - 
\alpha^2) (3 + \alpha^2)^{3/2} \Big)^{1/3}.
\end{align*}
By Theorem~\ref{thm:known_properties}\,\ref{it:relation_between_maps},
we have $L = \{ w \in \C : \abs{w - a_1}^2 \abs{w - a_2} \leq
(1-\alpha^2)^2/8 \}$. In order to compute 
$w = \Phi(z)$ for $z \in \C \setminus E$, we solve the equation
\begin{equation*}
2 p_n (w - a_1)^2 (w - a_2) = P(z) + \sqrt{P(z)^2 - 1};
\end{equation*}
see Theorem~\ref{thm:known_properties}\,\ref{it:relation_between_maps}. We use the mapping properties of~$\Phi$
established in Theorem~\ref{thm:mapping_properties_Phi} to determine the 
correct value of $w$. The image of $z_1$ is $w_1=(n_2 a_1 + n_1 a_2)/n$. For 
real $z$, if $z > 1$ then $w \in \R$ with $w > a_2$, if $z_1 < z < 
b_3$ then $w_1 < w < a_2$, if $b_2 < z < z_1$ then $a_1 < w < w_1$, and if $z < 
-1$ then $w < a_1$. If $\im(z) > 0$, we choose $w$ with $\im(w) > 0$ that is 
closest to $z$. For the lower half-plane, we use $\Phi(\conj{z}) = 
\conj{\Phi(z)}$.

In the limiting case $\alpha \to 1$, we have $b_2 \to -1$, $b_3 \to 1$,
$E$ degenerates to the set $\{ -1, 1 \}$, and $a_1 
\to -1$, $a_2 \to 1$. In the other limiting case $\alpha \to 0$, we have 
$b_2 \to \frac{1}{2}$ and $b_3 \to \frac{1}{2}$, so that $E$ tends to 
the interval $\cc{-1, 1}$, while $a_1 \to - 1/(2 \sqrt[3]{4}) \eqcolon 
a_1(0)$ and $a_2 \to 1/\sqrt[3]{4} \eqcolon a_2(0)$, so that the 
corresponding set $L$ ``converges'' to
$\{ w \in \C : \abs{w + 1/(2 \sqrt[3]{4})}^2 \abs{w - 
1/\sqrt[3]{4}} \leq 1/8 \}$,
which is an 
``eight'' self-intersecting at $w_1 = (n_2 a_1(0) + n_1 a_2(0))/n = 
1/(2\sqrt[3]{4})$, compare Figure~\ref{fig:P3_two_intervals}. However, the 
centers $a_1, a_2$ of $L$ \emph{do not} converge to the center $a_1 = 0$ 
of the lemniscatic domain $L = \{ w \in \C: \abs{w} \leq 1/2 \}$ corresponding 
to $E = \cc{-1, 1}$. This shows that a discontinuity in the connectivity (in our 
example from $\ell = 2$ components to $\ell = 1$ component) leads to a 
discontinuity in the centers.
\end{example}


If the polynomial~$P$ is real and even, then $E = P^{-1}(\cc{-1, 1})$ is
symmetric with respect to the real line as well as symmetric with respect to the origin,
and we obtain an explicit formula for the conformal map $\Phi$.


\begin{theorem} \label{thm:double_symmetry}
Let the polynomial $P(z) = \sum_{j=0}^n p_j z^j$ with $p_n \neq 0$ be real, 
even and assume that $P$ has exactly one critical point $z_1$ with critical 
value $P(z_1) \notin \cc{-1, 1}$ and $z_1$ is a simple zero of $P'$, that is 
$z_1 \notin E \coloneq P^{-1}(\cc{-1, 1})$ while the other $n-2$ critical 
points are in $E$. Then the following assertions hold:
\begin{enumerate}
\item The set $E$ consists of two components, $E = E_1 \cup E_2$, where $E_1, 
E_2$ are simply connected disjoint infinite compact sets with $E_1 = - E_2$, 
and $E = -E$, $E^* = E$, and $0 \notin E$.
In particular, only the following two cases can occur:
$E_1^* = E_1$, $E_2^* = E_2$ (case 1), and $E_1^* = E_2$ (case 2).

\item \label{it:double_sym_L}
If $E_1^* = E_1$, $E_2^* = E_2$, then
\begin{equation} \label{eqn:double_symmetry_a2}
a_2 = \left( \frac{(-1)^{n/2}}{2 p_n} \left( p_0 + \sqrt{p_0^2 - 1} \right)
\right)^{1/n} > 0
\end{equation}
and if $E_1^* = E_2$ then
\begin{equation}
a_2 = \ii \left( \frac{1}{2 p_n} \left( p_0 + \sqrt{p_0^2 - 1} \right) 
\right)^{1/n},
\end{equation}
with the positive real $n$-th root in both cases.
For $\sqrt{p_0^2 - 1}$, the positive branch is taken if $p_0 > 1$ and the 
negative branch if $p_0 < -1$.
Moreover, in both cases, $a_1 = - a_2$ and
\begin{equation}
L = \{ w \in \C : \abs{w^2 - a_2^2}^{1/2} \leq \capacity(E) = (2 
\abs{p_n})^{-1/n} \}.
\end{equation}

\item \label{it:double_sym_Phi}
The Walsh map of $E$ is $\Phi : \comp{E} \to \comp{L}$,
\begin{equation*}
\Phi(z)
= \sqrt{a_2^2 + \left( \frac{P(z)}{2 p_n} + \sqrt{\left( \frac{P(z)}{2 
p_n} \right)^2 - \frac{1}{4 p_n^2}} \right)^{2/n}}
\end{equation*}
with that branch of the $n$-th root and of the outer square root which yields 
positive real values of $\Phi(z)$ for sufficiently large $z$ on the positive 
real line.
\end{enumerate}
\end{theorem}

\begin{proof}
(i) The fact that $E$ has $\ell = 2$ components is an immediate consequence of
Theorem~\ref{thm:known_properties}\,(i), since $P$ has exactly one critical 
point outside $E$.  Thus, $E = E_1 \cup E_2$ with disjoint, simply connected, 
infinite compact sets $E_1, E_2$.
Since $P$ is real and even, we have $E^* = E$ and $E = -E$.
By Theorem~\ref{thm:symmetric_E}\,(ii) and Corollary~\ref{cor:E_symm_wrt_R},
assertion~(i) follows.

\ref{it:double_sym_L} and~\ref{it:double_sym_Phi}:
The number of zeros of $P$ in $E_1$ and in $E_2$ is the same, i.e., 
$n_1 = n_2 = n/2$.
Since $P$ is even, $z = 0$ is a critical point of $P$.
Since $0 \notin E$, we have $z_1 = 0$ and thus $p_0 = P(0) \in \R \setminus 
\cc{-1, 1}$.

We consider the two cases pointed out in~(i).

Case 1: If $E_1^* = E_1$ and $E_2^* = E_2$ then~\ref{it:double_sym_L}
is a direct 
consequence of Theorem~\ref{thm:aj_for_two_components} (where $p_{n-1} = 0$ 
since $P$ is even).
By~\eqref{eqn:Q}, $Q(w) = 2 p_n (w^2 - a_2^2)^{n/2}$ and 
\ref{it:double_sym_Phi} follows from Theorem~\ref{thm:known_properties}\,(iii).
Note that $P(z)/(2 p_n)$ is a real polynomial with positive leading 
coefficient $1/2$.  Therefore, the complex roots have to be taken as indicated 
in the theorem.

Case 2: $E_1^* = E_2$.  We reduce this case to case 1 as follows.  The polynomial $\widetilde{P}(z) \coloneq P(\ii z)$ also satisfies the assumptions of the theorem, and the set
$\widetilde{E} = \widetilde{P}^{-1}(\cc{-1, 1}) ) = - \ii E$ 
falls under case~1, so that the corresponding set $\widetilde{L}$ and
conformal map $\widetilde{\Phi} :\comp{\widetilde{E}} \to \comp{\widetilde{L}}$
are determined by the formulas in case~1.
By~\cite[Lem.~2.3]{SeteLiesen2016}, we have $a_2 = \ii \widetilde{a}_2$ and
$\Phi(z) = \ii \widetilde{\Phi}(- \ii z)$, which yields after a short 
calculation the formulas in case~2.
\end{proof}

\begin{remark}  \label{rem:double_symmetry_cases}
\begin{enumerate}
\item \label{it:double_symmetry_case1}
In Theorem~\ref{thm:double_symmetry}, the following equivalence holds:
$E$ contains a real point if and only if $E_1^* = E_1$ and $E_2^* = E_2$,
which follows from~\cite[Lem.~A.2]{SchiefermayrSete2022}.
In more detail:
If $x \in E \cap \R$, then without loss of generality, $x \in E_1$ and hence
$x \in E_1^*$, hence $E_1 = E_1^*$.
Conversely, if $E_1^* = E_1$ and $E_2^* = E_2$, both $E_1$ and $E_2$ contain at 
least one real point~\cite[Lem.~A.2]{SchiefermayrSete2022}.

\item If $E_1^* = E_1$, $E_2^* = E_2$ in Theorem~\ref{thm:double_symmetry} then 
$\partial L \cap \R$
consists of the four points $c_1 < c_2 < 0 < c_3 < c_4$ with $c_1 = - c_4$, $c_2 = - c_3$ and
\begin{equation*}
c_{3,4} = \left( a_2^2 \mp (2 \abs{p_n})^{-2/n} \right)^{1/2}.
\end{equation*}
Denote $E_1 \cap \R = \cc{b_1, b_2}$ and $E_2 \cap \R = \cc{b_3, b_4}$ with $b_1 \leq b_2 < 0 < b_3 \leq b_4$; 
see~\cite[Lem.~A.2]{SchiefermayrSete2022}.  Since $E_1 = -E_2$, we have
$b_1 = -b_4$, $b_2 = -b_3$.
Then, by Theorem~\ref{thm:mapping_properties_Phi}\,(iv) and~\ref{it:known_properties_crit_pts},
$\Phi(\oo{-\infty, b_1}) = \oo{-\infty, c_1}$,
$\Phi(\oo{b_2, 0}) = \oo{c_2, 0}$,
$\Phi(\oo{0, b_3}) = \oo{0, c_3}$,
$\Phi(\oo{b_4, \infty}) = \oo{c_4, \infty}$.
In addition, $\Phi$ satisfies $\Phi(0) = 0$, $\Phi(\ii \R^-) = \ii \R^-$,
$\Phi(\ii \R^+) = \ii \R^+$.

\item If $E_1^* = E_2$, the roles of $\R$ and $\ii \R$ in (ii) switch.
\end{enumerate}
\end{remark}

Let us consider two illustrative examples for Theorem~\ref{thm:double_symmetry}.

\begin{example}[Two symmetric real intervals] \label{ex:two_sym_real_intervals}
For $0 < \alpha < \beta$, consider the polynomial
\begin{equation*}
P(z) = \frac{2}{\beta^2 - \alpha^2} (z^2 - \beta^2) + 1,
\end{equation*}
of degree $n = 2$.  Then
$E = P^{-1}(\cc{-1, 1}) = \cc{-\beta, -\alpha} \cup \cc{\alpha, \beta}$, and
$n_1 = n_2 = 1$.  The critical point of $P$ is $z_1 = 0 \notin E$, so that 
$P$ satisfies the assumptions in Theorem~\ref{thm:double_symmetry}.
Since $p_0 = - \frac{\beta^2 + \alpha^2}{\beta^2 - \alpha^2} < -1$, we have
$p_0 + \sqrt{p_0^2 - 1} = - \frac{\beta + \alpha}{\beta - \alpha}$.
Hence, by~\eqref{eqn:double_symmetry_a2},
$a_2 = (\beta + \alpha)/2$ and $a_1 = -a_2$, so that
$L = \{ w \in \C : \abs{w - a_2^2}^{1/2} \leq \sqrt{(\beta^2 - \alpha^2)/4} \}$.
Since
\begin{equation*}
P(z) + \sqrt{P(z)^2 - 1} = \frac{1}{\beta^2 - \alpha^2}
\Big( 2 z^2 - \alpha^2 - \beta^2 + 2 \sqrt{(z^2 - \alpha^2)(z^2 - \beta^2)} 
\Big),
\end{equation*}
we obtain
\begin{equation*}
\Phi(z)
= \frac{1}{\sqrt{2}} \sqrt{z^2 + \alpha \beta + \sqrt{(z^2 - \alpha^2)(z^2 - 
\beta^2)}}
\end{equation*}
which is in accordance with~\cite[Cor.~3.3]{SeteLiesen2016}.
Figure~\ref{fig:two_sym_real_intervals} shows the sets $E$ and $L$ and the 
image of a grid under $\Phi$.  For the numerical evaluation of $\Phi(z)$ we use 
the modified formula
\begin{equation*}
\Phi(z) = z \sqrt{\left(a_2^2 + \frac{P(z) + \sqrt{P(z)^2 - 1}}{2 p_n} \right) / 
z^2}, \quad z \in \C \setminus E, z \neq 0,
\end{equation*}
with the main branch of the square root, and $\Phi(0) = 0$; compare 
to~\cite[Thm.~3.1]{SeteLiesen2016}.
\end{example}

\begin{figure}
{\centering
\includegraphics[width=0.48\linewidth]{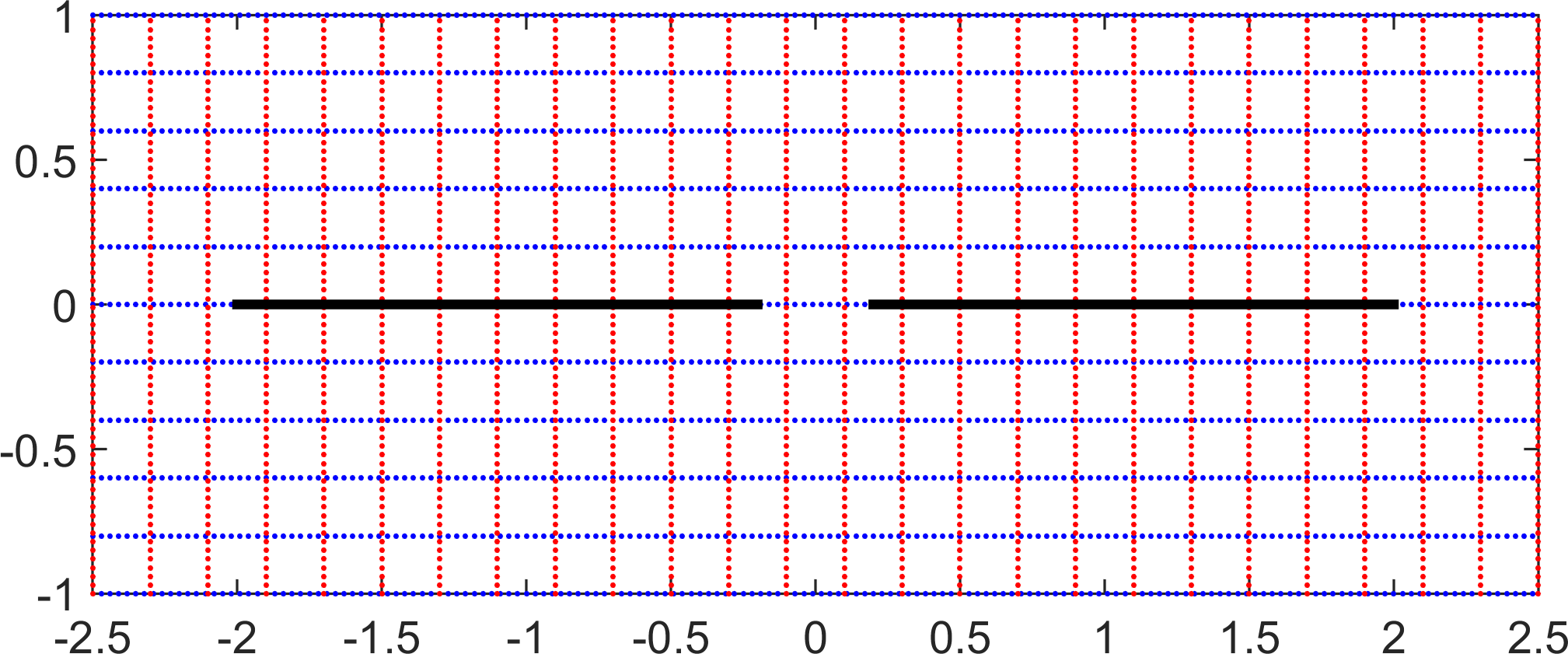}
\includegraphics[width=0.48\linewidth]{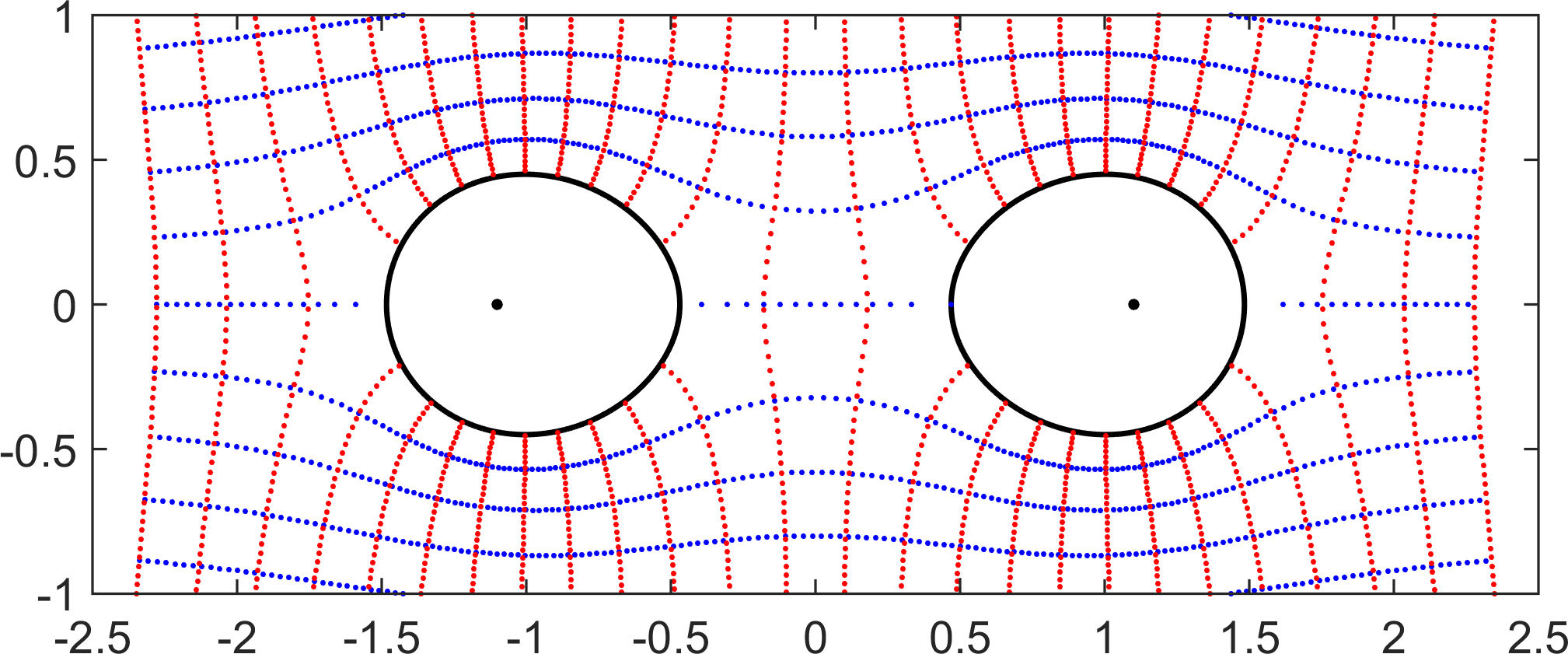}

}
\caption{$E = P^{-1}(\cc{-1, 1}) = \cc{-2, -0.2} \cup \cc{0.2, 2}$ 
in Example~\ref{ex:two_sym_real_intervals}.
Left: $E$ (black) and a grid.
Right: $\partial L$ (black), $a_1, a_2$ (black dots), and the image of the grid under $\Phi$.}
\label{fig:two_sym_real_intervals}
\end{figure}

\begin{example} \label{ex:intersecting_arcs}
Consider the polynomial
\begin{equation*}
P(z) = (z^2 - \alpha^2)^2
\end{equation*}
of degree $n=4$ with parameter $\alpha > 1$ and $E = P^{-1}(\cc{-1, 1})$.
The critical points of $P$ are $z_1 = 0$ and $z_{2,3} = \pm \alpha$ with the
critical values $P(z_1) = \alpha^4 > 1$ and $P(z_{2,3}) = 0$, hence
$z_1 = 0 \notin E$ and $z_{2,3} \in E$ and the assumptions of
Theorem~\ref{thm:double_symmetry} are satisfied.
Since $\pm \alpha \in E$, by Theorem~\ref{thm:double_symmetry} and
Remark~\ref{rem:double_symmetry_cases}\,\ref{it:double_symmetry_case1},
we have the case $E_1^* = E_1$, $E_2^* = E_2$.
Note that $E_1 = - E_2$.
The component $E_2$ is the union of the interval
$\cc{\sqrt{\alpha^2 - 1}, \sqrt{\alpha^2 + 1}}$ and of a Jordan arc symmetric
with respect to the real line with endpoints $b_3 = 
\sqrt{\alpha^2 + \ii}$ and $b_4 = \conj{b}_3 = \sqrt{\alpha^2 - \ii}$
(where $\re(b_3) > 0$) intersecting the interval at the critical point
$\alpha$; see~\cite[Thm.~1]{Schiefermayr2012}.
By~\eqref{eqn:double_symmetry_a2}, the points $a_1, a_2$ are given by
\begin{equation*}
a_2 = \left( \frac{1}{2} \bigl( \alpha^4 + \sqrt{\alpha^8 - 1} \bigr) 
\right)^{1/4},
\quad a_1 = - a_2.
\end{equation*}
The conformal map is given by
\begin{equation*}
\Phi(z)
= \frac{1}{2^{1/4}}
\sqrt{\Bigl( \alpha^4 + \sqrt{\alpha^8 - 1} \Bigr)^{1/2}
+ \Bigl( (z^2 - \alpha^2)^2 + \sqrt{(z^2 - \alpha^2)^4 -1} \Bigr)^{1/2}}.
\end{equation*}
See Figure~\ref{fig:intersecting_arcs} for an illustration.
\end{example}

\begin{figure}
{\centering
\includegraphics[width=0.48\linewidth]{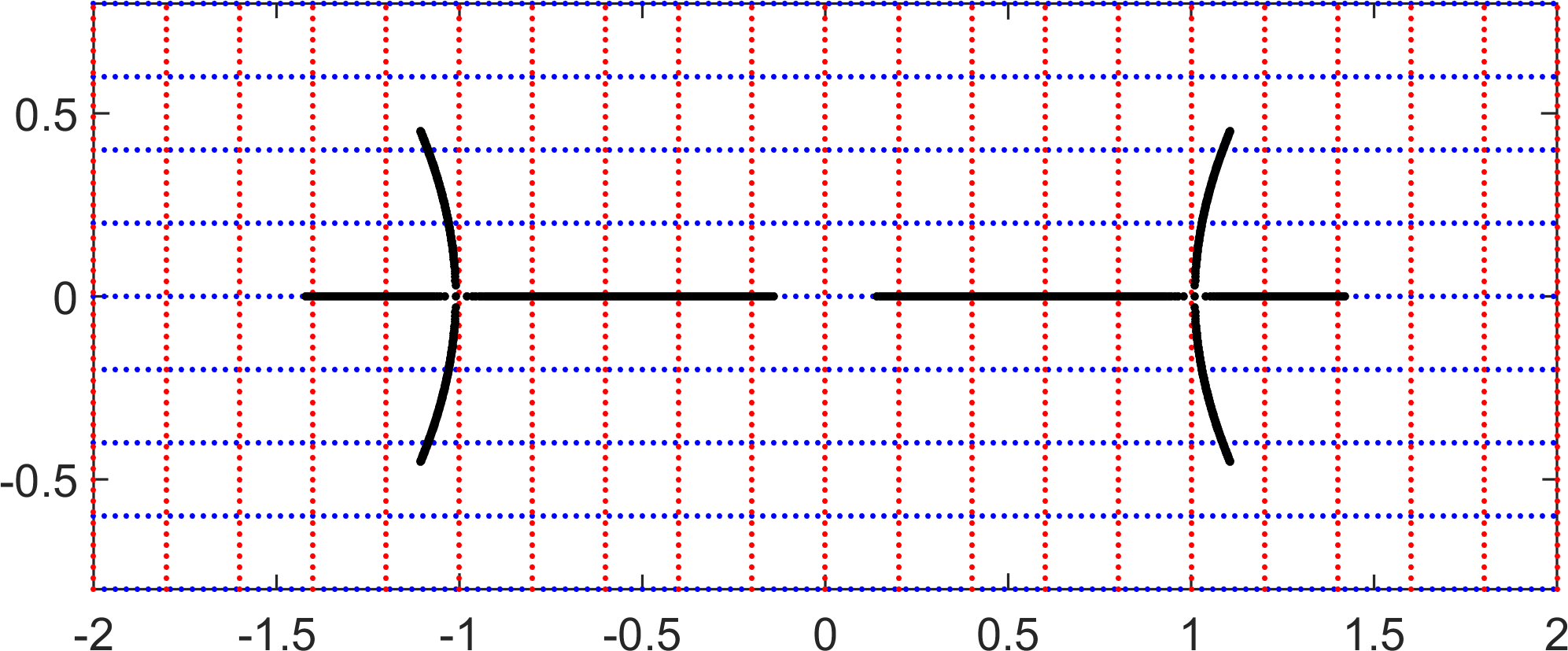}
\includegraphics[width=0.48\linewidth]{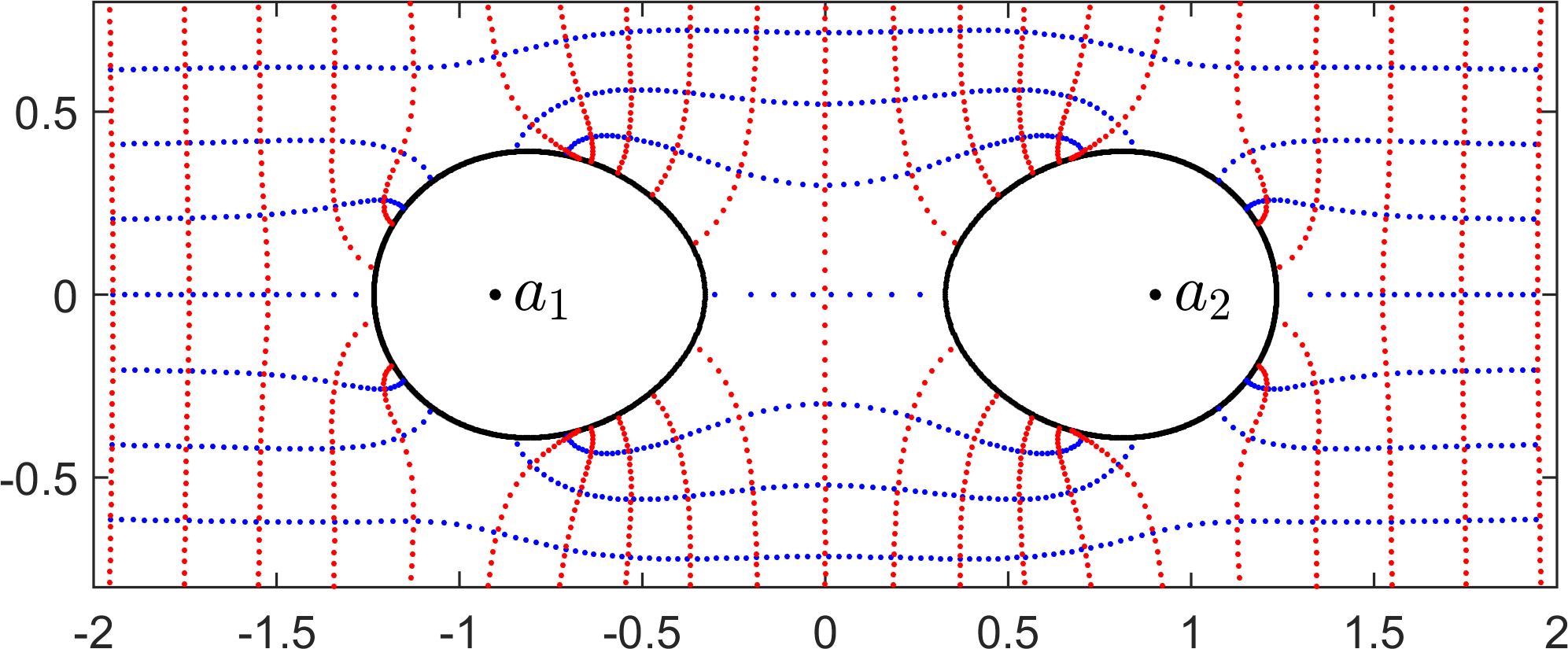}

}
\caption{Pre-image $E = P^{-1}(\cc{-1, 1})$ in 
Example~\ref{ex:intersecting_arcs} with $\alpha = 1.01$.
Left: $E$ (black lines) and a grid.
Right: $\partial L$ (black curves), $a_1, a_2$ (black dots), and the image of the grid under $\Phi$.
}
\label{fig:intersecting_arcs}
\end{figure}

\section{Sets with Three components}\label{sect:three_components}

Let the polynomial $P$ be given as in~\eqref{eqn:P} and assume that $E$ given 
in~\eqref{eqn:poly_preimage} has $\ell = 3$ components, that is, $E = P^{-1}(\cc{-1, 1}) = E_1 \cup 
E_2 \cup E_3$. Then $Q(w)$ in~\eqref{eqn:Q} has the form
\begin{equation} \label{eqn:Q_ell3}
Q(w) = 2 p_n (w - a_1)^{n_1} (w - a_2)^{n_2} (w - a_3)^{n_3}
\end{equation}
with $n_1 + n_2 + n_3 = n$. Moreover, $Q$ has two critical points $w_{1,2}$ in 
$\C \setminus L$ which are the solutions of
\begin{equation}
n_1 (w - a_2) (w - a_3) + n_2 (w - a_1) (w - a_3) + n_3 (w - a_1) (w - a_2) = 0,
\end{equation}
see Theorem~\ref{thm:known_properties}\,\ref{it:known_properties_crit_pts}.  A short computation shows that
\begin{equation}
\begin{split}
w_{1,2} &= \frac{1}{2 n} \bigg( (n_2 + n_3) a_1 + (n_1 + n_3) a_2 + (n_1 + 
n_2) a_3 \\
& \mp \sqrt{ \bigl( (n_2 + n_3) a_1 + (n_1 + n_3) a_2 + (n_1 + n_2) a_3 
\bigr)^2 
- 4 n (n_3 a_1 a_2 + n_2 a_1 a_3 + n_1 a_2 a_3)} \Big).
\end{split} \label{eqn:crit_pts_Q_ell3}
\end{equation}

If all three components $E_1, E_2, E_3$ of $E$ are symmetric with respect to 
$\R$, then
$a_1, a_2, a_3$ are the solution of the non-linear system of equations
\eqref{eqn:ell_relation_crit_pts}--\eqref{eqn:ell_sum_nj_aj}
in Theorem~\ref{thm:ell_components_real_symmetry}, which can be solved
numerically; see the following example.


\begin{example} \label{ex:three_intervals}
Let us construct a polynomial pre-image $E = P^{-1}(\cc{-1, 1})$ consisting
of three real intervals.  By Theorem~\ref{thm:E_ell_intervals}, the set
\begin{equation} \label{eqn:three_intervals_E}
E = \cc{-1, \gamma_1 - \alpha} \cup \cc{\gamma_1 + \alpha, \gamma_2 - \beta} \cup \cc{\gamma_2 + 
\beta, 1},
\end{equation}
is the polynomial pre-image of a polynomial $P$ of degree $n = 3$ if 
and only if $\alpha, \beta, \gamma_1, \gamma_2$ satisfy the equations
\begin{equation}
(-1)^k - (\gamma_1 - \alpha)^k - (\gamma_1 + \alpha)^k + (\gamma_2 - \beta)^k + 
(\gamma_2 + \beta)^k - 1^k = 0
\end{equation}
for $k = 1$ and $k = 2$.  Simplifying this system gives
\begin{equation} \label{eqn:three_intervals_c1_c2}
\gamma_1 = \frac{1}{2} (\alpha^2 - \beta^2 - 1)
\quad \text{and} \quad
\gamma_2 = \frac{1}{2} (\alpha^2 - \beta^2 + 1).
\end{equation}
Hence, for $\alpha, \beta > 0$, $\alpha + \beta < 1$, and $\gamma_1, \gamma_2$
given by~\eqref{eqn:three_intervals_c1_c2}, the 
set~\eqref{eqn:three_intervals_E} is a polynomial pre-image.
By~\eqref{eqn:P_for_E_ell_intervals}, the polynomial of degree $n = 3$ (unique 
up to sign) with $E = P^{-1}(\cc{-1, 1})$ is given by
\begin{equation*}
\begin{aligned}
P(z)
&= -1 - \frac{\bigl((z - \gamma_1)^2 - \alpha^2 \bigr) (z - 1)}{(1 + \gamma_1)^2 - \alpha^2} \\
&= -\frac{z^3 + (\beta^2 - \alpha^2) z^2 + (\gamma_1^2 - \beta^2 - 1) z + \alpha^2 - \beta^2}{(1+\gamma_1)^2 - \alpha^2}.
\end{aligned}
\end{equation*}
In particular,
\begin{equation*}
p_3 = - \frac{1}{(1 + \gamma_1)^2 - \alpha^2}, \quad
p_2 = - \frac{\beta^2 - \alpha^2}{(1 + \gamma_1)^2 - \alpha^2}.
\end{equation*}
The critical points of $P$ are given by
\begin{equation*}
z_{1,2} = \frac{1}{6} \bigl( 2 \alpha^2 - 2 \beta^2 \mp \sqrt{\alpha^4 - 2 
\alpha^2 (\beta^2 - 3) + (\beta^2 + 3)^2} \bigr).
\end{equation*}
Next, we compute $a_1, a_2, a_3$ for the set in~\eqref{eqn:three_intervals_E}.  
Note that $n_1 = n_2 = n_3 = 1$.  
Using~\eqref{eqn:Q_ell3} and~\eqref{eqn:crit_pts_Q_ell3}, we can solve
the system~\eqref{eqn:ell_relation_crit_pts}--\eqref{eqn:ell_sum_nj_aj}
numerically for $a_1$, $a_2$, $a_3$.
Choosing numerical values, say $\alpha = 0.05$ and $\beta = 0.3$, yields
\begin{equation*}
E = [ -1, -\tfrac{19}{20}] \cup [-\tfrac{79}{20}, \tfrac{5}{32}] \cup 
[\tfrac{121}{160}, 1]
\end{equation*}
and
\begin{equation*}
a_1 = -0.7751\ldots, \quad
a_2 = -0.1648\ldots, \quad
a_3 = 0.8525\ldots.
\end{equation*}
The \emph{Mathematica} command \texttt{NSolve} returns six distinct triples
$(a_1, a_2, a_3)$, which are permutations of each other, where only one
satisfies $a_1 < a_2 < a_3$.
Figure~\ref{fig:three_intervals} illustrates the conformal map $\Phi$.
We compute $w = \Phi(z)$ by solving $Q(w) = P(z) + \sqrt{P(z)^2 - 1}$ and
determine
the correct solution $w$ (out of the $n = 3$ solutions) using the mapping
properties of $\Phi$ in Theorem~\ref{thm:mapping_properties_Phi}.
\end{example}

\begin{figure}
{\centering
\includegraphics[width=0.48\linewidth]{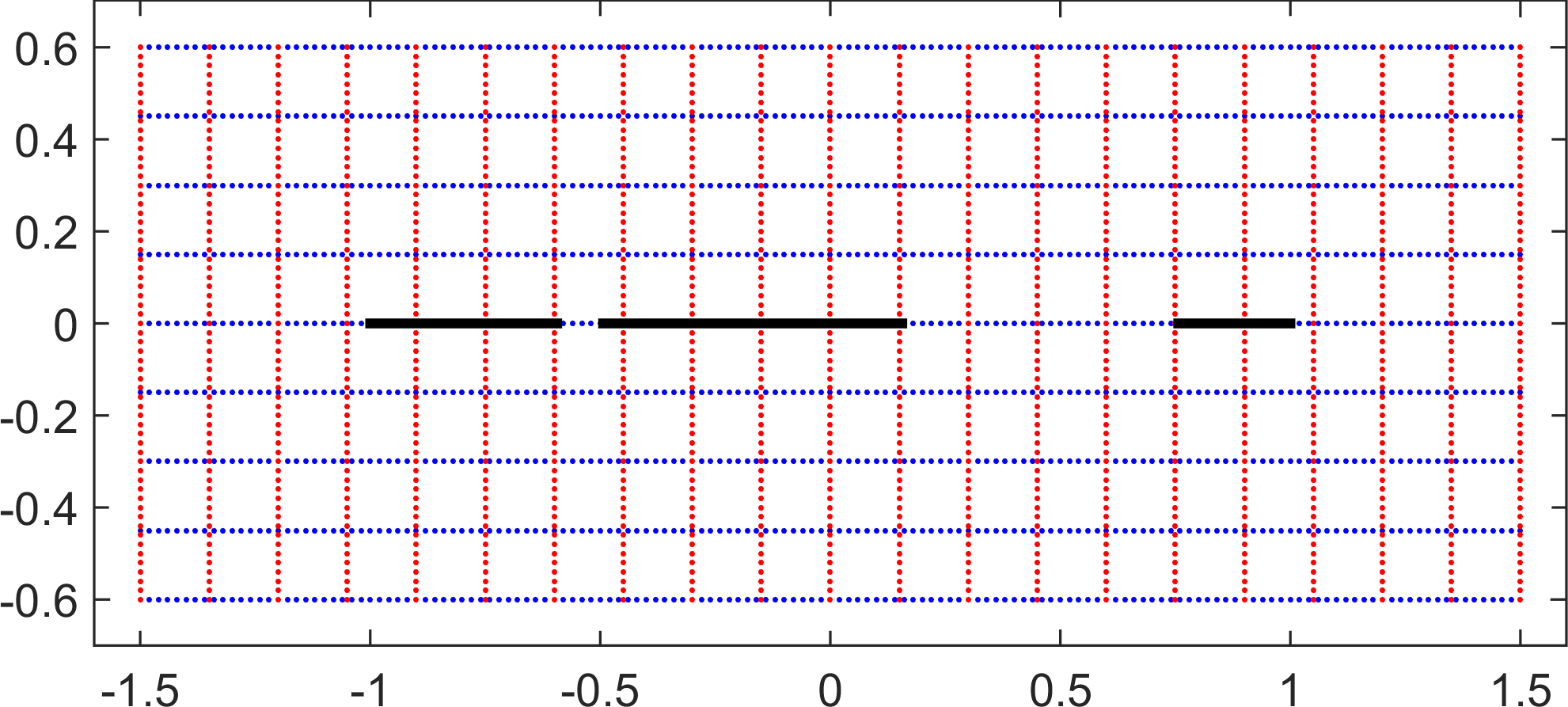}
\includegraphics[width=0.48\linewidth]{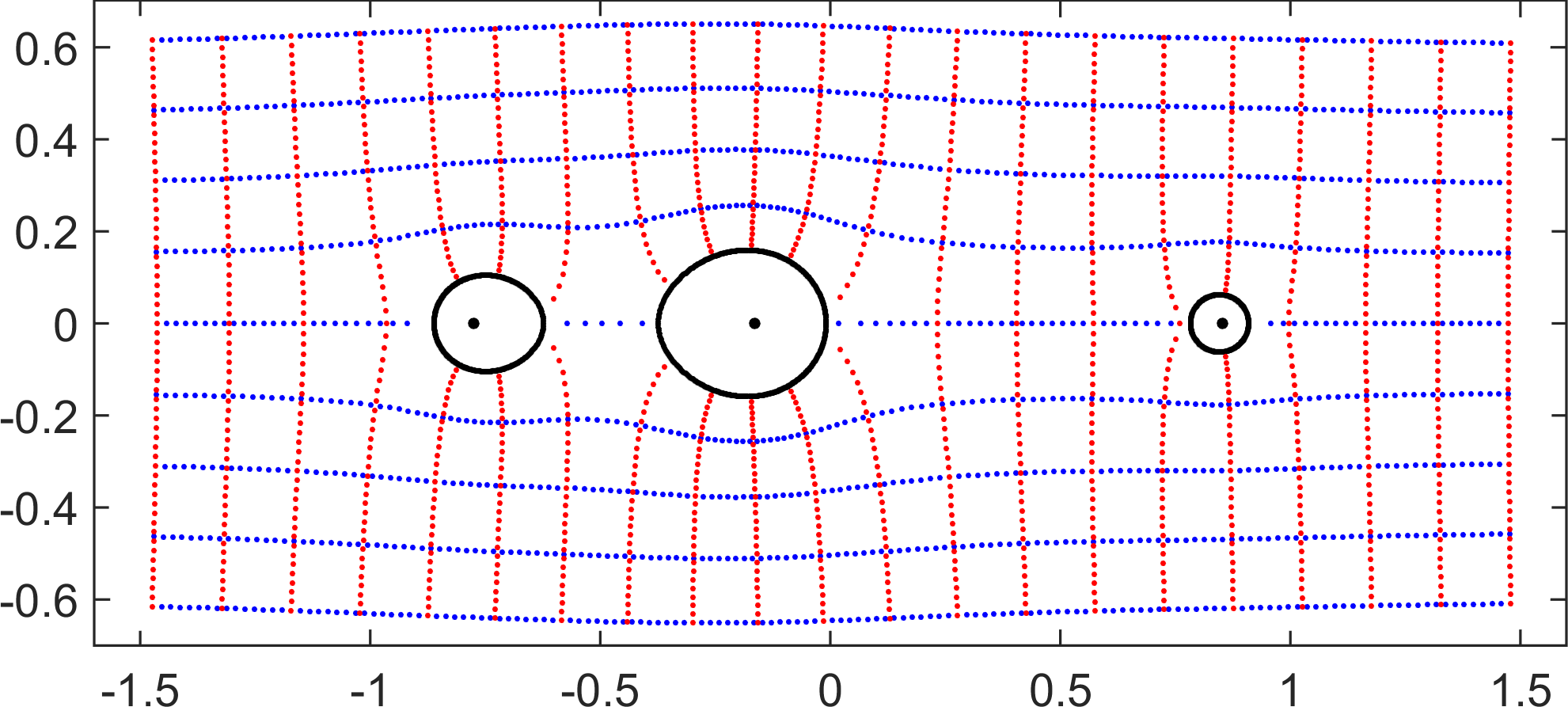}

}
\caption{Left: The set $E$ in Example~\ref{ex:three_intervals} with $\alpha = 
0.05$ and $\beta = 0.3$ in black and a grid.
Right: Corresponding lemniscatic domain ($\partial L$ in black), the points 
$a_1, a_2, a_3$ (black dots), and the image of the grid under $\Phi$.}
\label{fig:three_intervals}
\end{figure}

If $E$ has ``double symmetry'', that is, $E$ is symmetric with respect to zero
and each component is symmetric with respect to $\R$, then the centers
$a_1$, $a_2$, $a_3$ can be given explicitly.


\begin{theorem} \label{thm:ell3_double_symmetry}
Let $E = P^{-1}(\cc{-1, 1}) = E_1 \cup E_2 \cup E_3$ with $E = -E$ and $E_j^* = 
E_j$ for $j = 1, 2, 3$. Let $z_1 < z_2$ be the critical points of $P$ in $\R 
\setminus E$. Then, $z_2 = -z_1$, $n_1 = n_3$, $a_1 < a_2 = 0 < a_3$ with $a_1 = -a_3$ and
\begin{equation} \label{eqn:a3_3components_double_sym}
a_3 = \sqrt{n} \sqrt[n]{
\frac{\abs{P(z_2) + \sqrt{P(z_2)^2 - 1}}}{2 \abs{p_n} n_2^{n_2/2} 
(2n_3)^{n_3}} }
\end{equation}
with the positive real $n$-th root.  In~\eqref{eqn:a3_3components_double_sym},
$z_2$ can also be replaced by $z_1$.
\end{theorem}

\begin{proof}
By Theorem~\ref{thm:symmetric_E}, $E_1 = -E_3$ and $E_2 = -E_2$.  By 
Theorem~\ref{thm:aj_symm_origin}, $z_2 = -z_1$, $n_1 = n_3$, $a_1 = -a_3$ and $a_2 = 0$. 
By~Theorem~\ref{thm:mapping_properties_Phi}\,(v), we have $a_1 < a_2 = 0 < 
a_3$. Therefore, by~\eqref{eqn:Q_ell3}, $Q(w) = 2 p_n w^{n_2} (w^2 - 
a_3^2)^{n_3}$, and, by~\eqref{eqn:crit_pts_Q_ell3}, $w_{1,2} = \mp a_3 
\sqrt{n_2/n}$, and
\begin{equation*}
Q(w_{1,2}) = 2 p_n \left( \mp \sqrt{\frac{n_2}{n}} \right)^{n_2}
\left ( - \frac{2 n_3}{n} \right)^{n_3} a_3^n.
\end{equation*}
Since $Q(w_{1,2}) = P(z_{1,2}) + \sqrt{P(z_{1,2})^2 - 1}$, see 
Theorem~\ref{thm:ell_components_real_symmetry}, and $a_3 > 0$, we obtain
\begin{equation*}
a_3 = \sqrt[n]{\left( \sqrt{\frac{n}{n_2}} \right)^{n_2}
\left( \frac{n}{2 n_3} \right)^{n_3} \frac{1}{2 \abs{p_n}} 
\abs{P(z_2) + \sqrt{P(z_2)^2 - 1}}},
\end{equation*}
which implies~\eqref{eqn:a3_3components_double_sym}.
\end{proof}

\begin{example} \label{ex:three_sym_real_intervals}
Let us construct a polynomial~$P$ of degree~$n=3$ with a pre-image consisting 
of three symmetric intervals 
$E = [-1,-\beta] \cup [-\alpha,\alpha] \cup [\beta,1]$ with $0 < \alpha < \beta 
< 1$.  By Theorem~\ref{thm:E_ell_intervals}, $E$ is a polynomial pre-image if 
and only~\eqref{eqn:system_bk} holds, which is equivalent to 
$\beta = 1-\alpha$.  This implies $0 < \alpha < 1/2$.
By Remark~\ref{rem:E_ell_intervals}\,\ref{it:E_ell_intervals_P},
\begin{equation}
P(z) = -1 + 2 \frac{(z + 1)(z - \alpha) (z - (1 - \alpha))}{(-\beta+1) 
(-\beta-\alpha) (-\beta - (1 - \alpha))}
= \frac{z^3 - (1-\alpha+\alpha^2) z}{\alpha (1-\alpha)}
\end{equation}
satisfies
$P^{-1}([-1,1]) = E = [-1,-(1-\alpha)]\cup[-\alpha,\alpha]\cup[1-\alpha,1]$ with 
$n_1 = n_2 = n_3 = 1$.
The critical points of~$P$ are $z_{1,2} =\mp\sqrt{(1-\alpha+\alpha^2)/3}$.
We have $a_2 = 0$, $a_1 = - a_3$ and, by~\eqref{eqn:a3_3components_double_sym},
\begin{equation*}
a_3 = \sqrt{3} \sqrt[3]{ \frac{\alpha (1-\alpha)}{4} \Bigl\lvert P(z_2) + 
\sqrt{P(z_2)^2 -1} \Bigr\rvert},
\end{equation*}
where a short calculation shows
\begin{equation*}
P(z_2) + \sqrt{P(z_2)^2 - 1} = \frac{2}{3 \sqrt{3}} \frac{(1-\alpha + \alpha^2)^{3/2}}{\alpha 
(1-\alpha)} + \frac{2 - 3 \alpha - 3 \alpha^2 + 2 \alpha^3}{3 \sqrt{3} \alpha 
(1-\alpha)}.
\end{equation*}
This yields the explicit formula
\begin{equation*}
a_3 = \sqrt[3]{\frac{1}{2} (1 - \alpha + \alpha^2)^{3/2} + \frac{1}{4} (2 - 3 
\alpha - 3 \alpha^2 + 2 \alpha^3)},
\end{equation*}
with positive real roots.  Since $a_2 = 0$, $a_1 = - a_3$, and $n_1 = n_2 = n_3 
= 1$, we have explicitly determined the canonical domain $\comp{L}$.

In the limiting case $\alpha \to 0$, the set $E$ degenerates to
$\{ -1, 0, 1 \}$ and $a_3 \to 1$.  In the other limiting case
$\alpha \to 1/2$, the set $E$ tends to the interval $\cc{-1, 1}$
and
\begin{equation*}
\lim_{\alpha \to 1/2} a_3 = \frac{\sqrt{3}}{2 \sqrt[3]{2}}
= 0.6874\ldots.
\end{equation*}
Thus, the set $L$ ``converges'' to
$\{ w \in \C : \abs{w} \cdot \abs{w^2 - 3 \cdot 2^{-8/3}} \leq 1/8 \}$,
whose boundary is a double figure eight self-intersecting at $\pm 2^{-4/3}$.
As observed in Example~\ref{ex:P3_two_intervals}, the centers of $L$
do not converge to the center $0$ of the lemniscatic domain
$\{ w \in \C : \abs{w} \leq 1/2 \}$ corresponding to $\cc{-1, 1}$.
This shows again that a discontinuity in the connectivity
(in this example from $\ell = 3$ components to $\ell = 1$ component)
leads to a discontinuity of the centers.

Figure~\ref{fig:three_sym_real_intervals} illustrates the sets $E$ and $L$,
and the conformal map $\Phi$ 
for several values of $\alpha$.  For $z \in \C \setminus E$, we determine $w 
= \Phi(z)$ by solving the polynomial equation $Q(w) = P(z) + \sqrt{P(z)^2 - 1}$,
where the correct value is determined with the help of the mapping properties of 
$\Phi$ in Theorem~\ref{thm:mapping_properties_Phi}.
\end{example}

\begin{figure}[t!]
{\centering
\includegraphics[width=0.48\linewidth]
{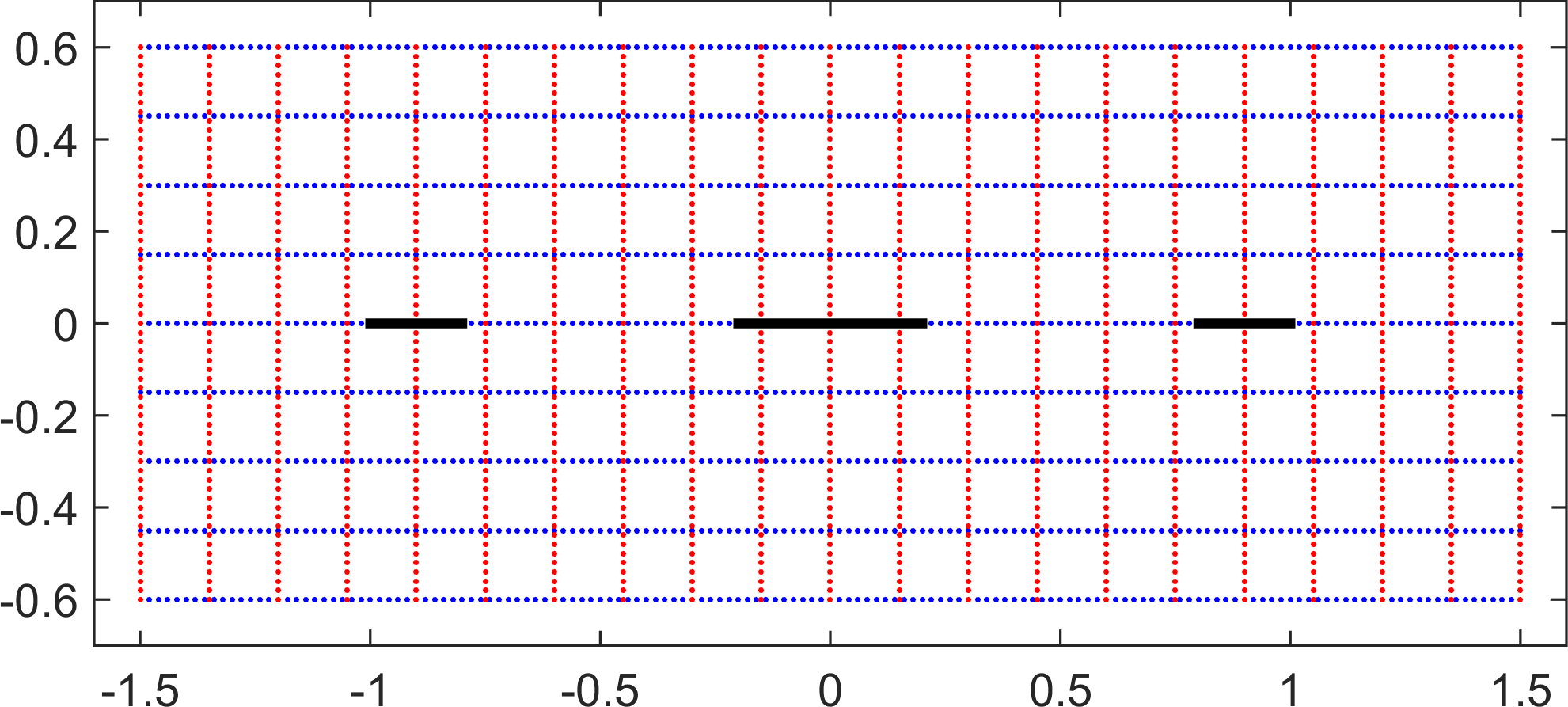}
\includegraphics[width=0.48\linewidth]
{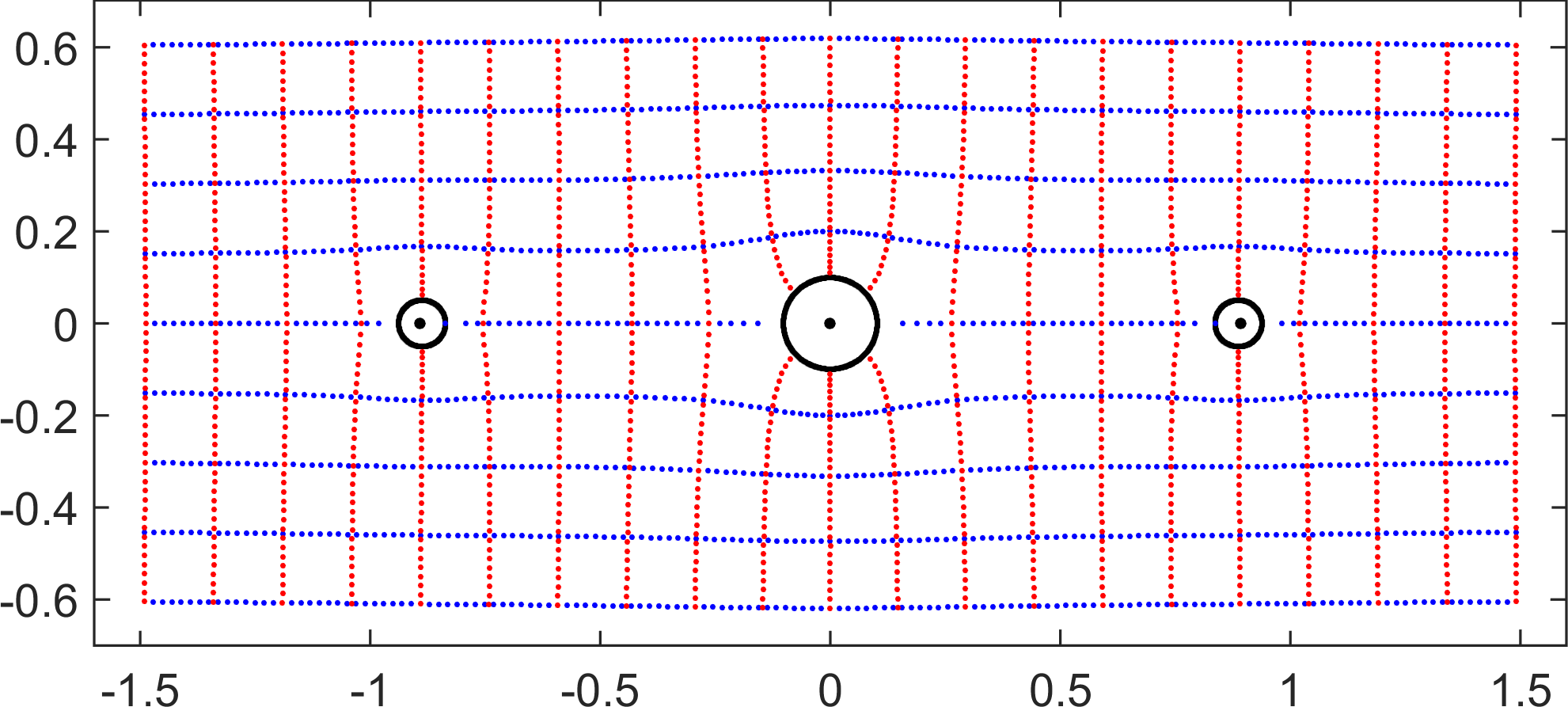}

\includegraphics[width=0.48\linewidth]
{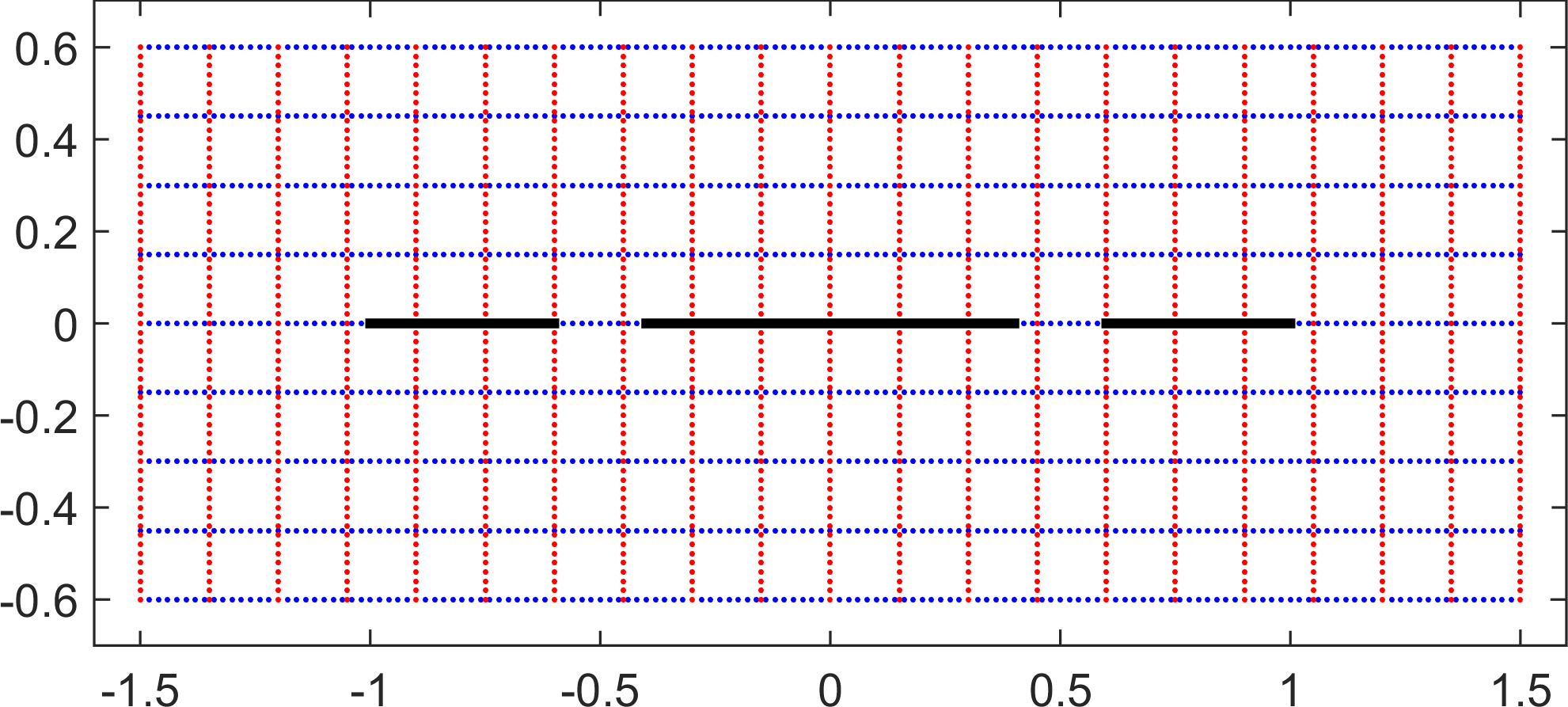}
\includegraphics[width=0.48\linewidth]
{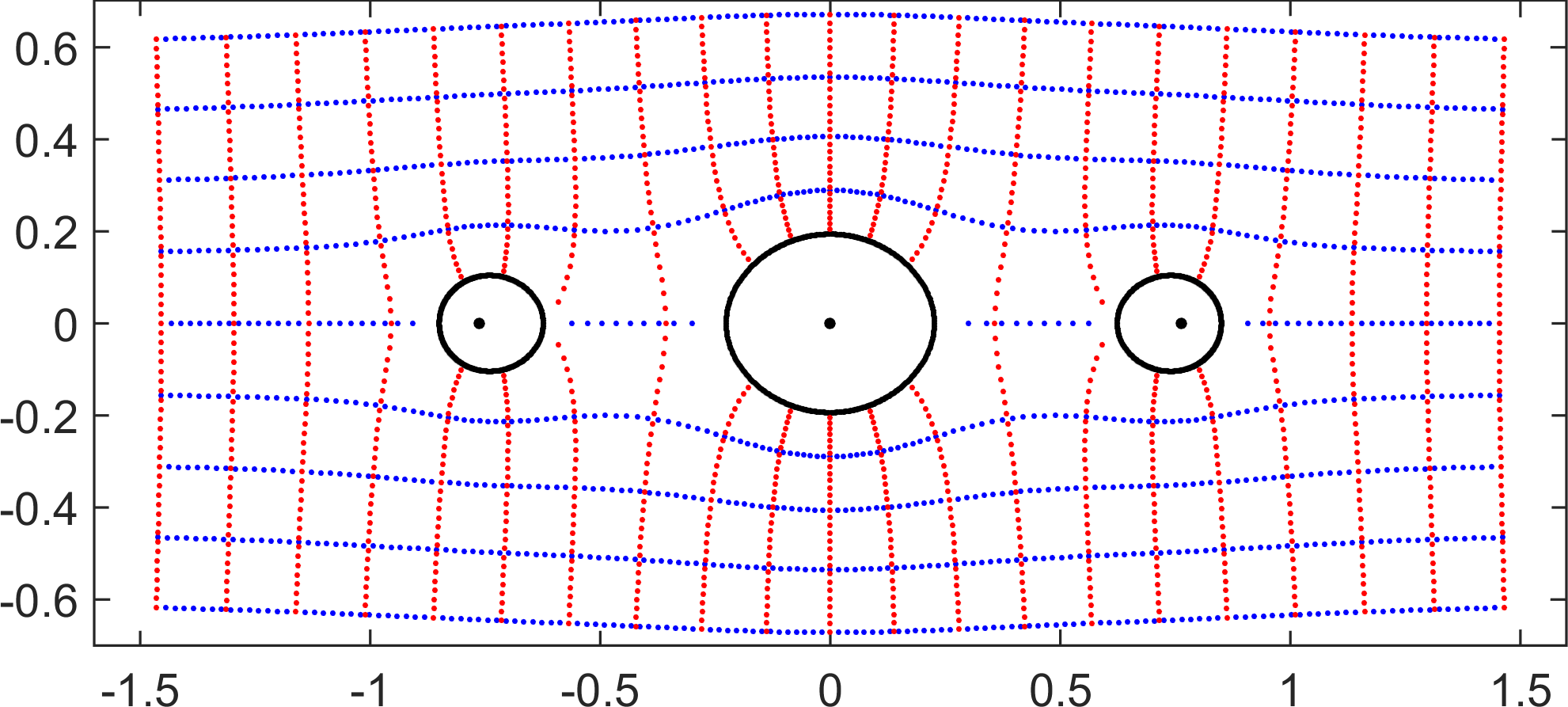}

\includegraphics[width=0.48\linewidth]
{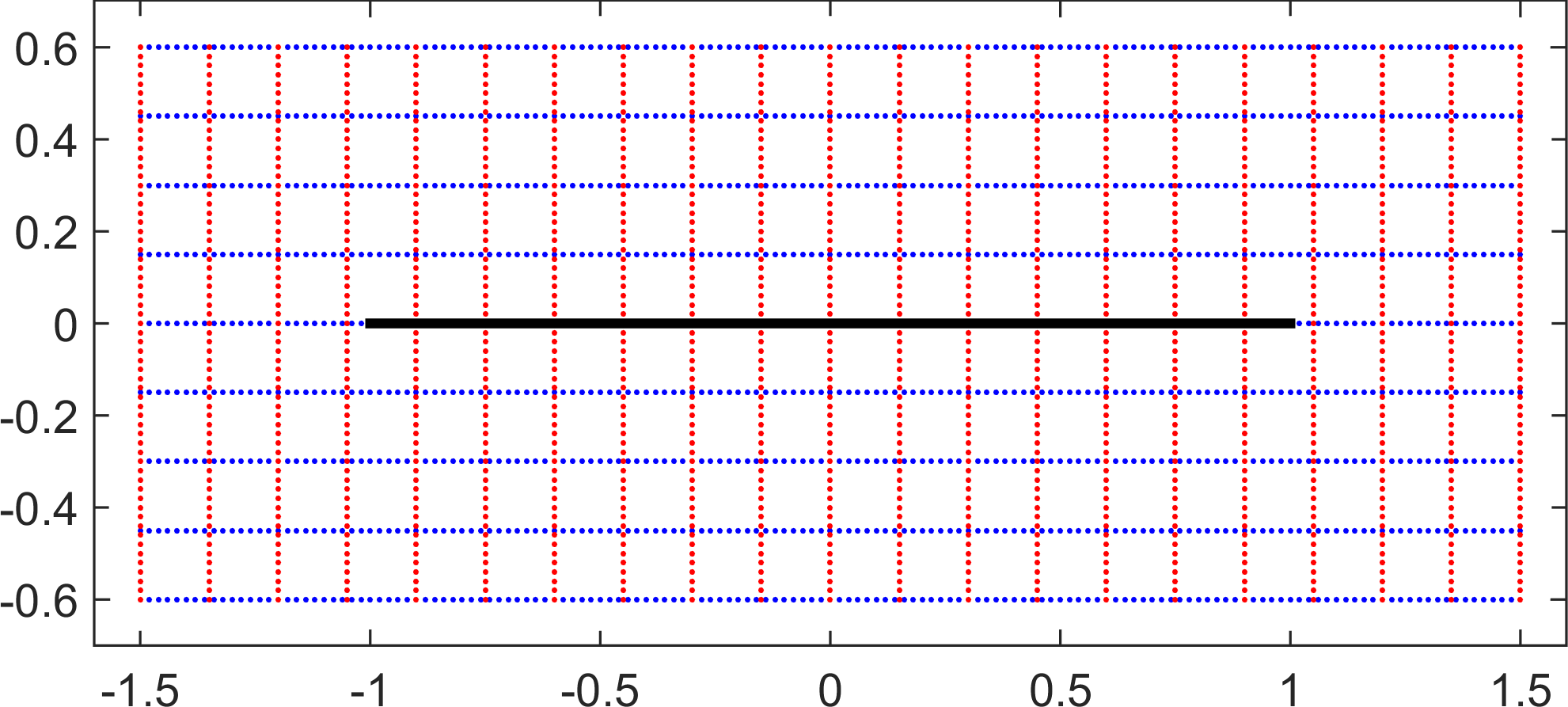}
\includegraphics[width=0.48\linewidth]
{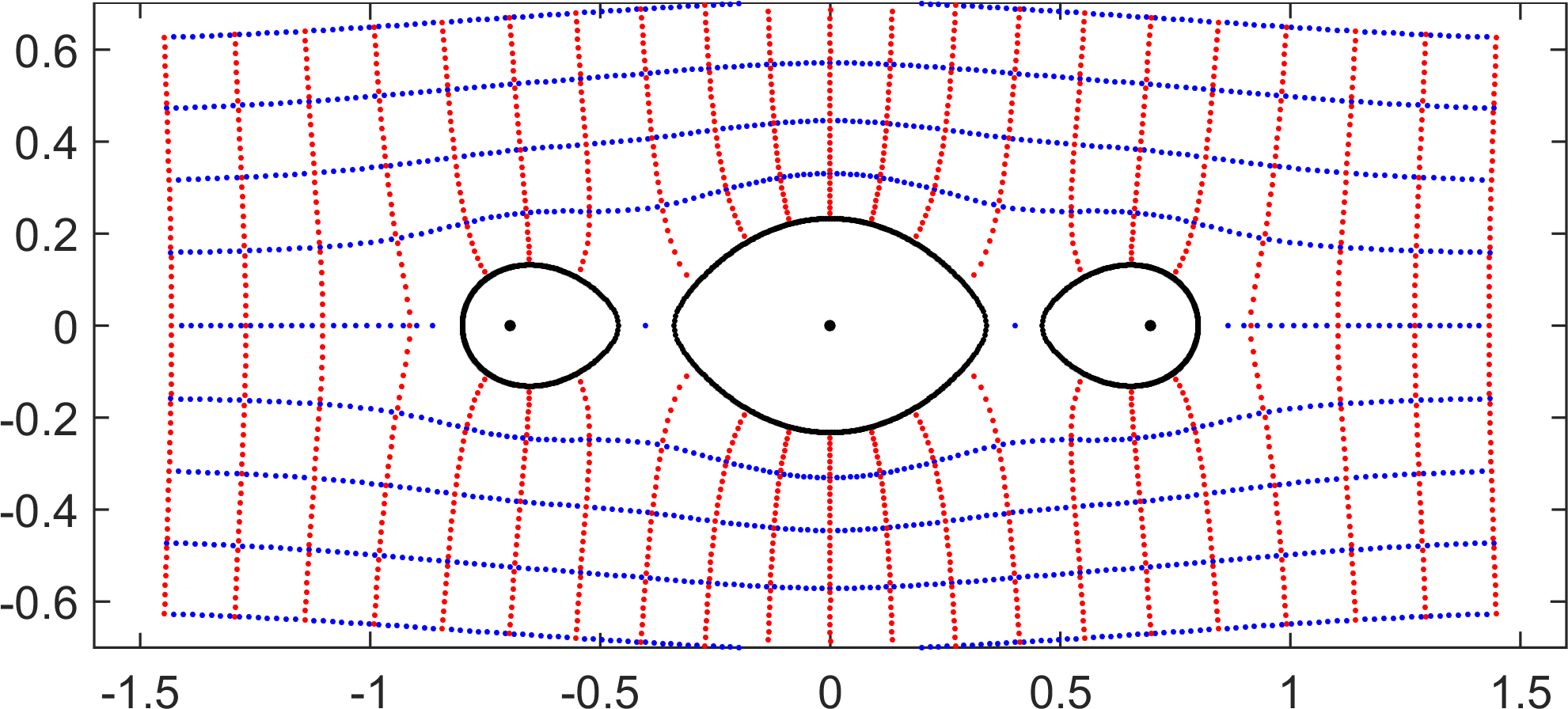}

}
\caption{Set $E$ in Example~\ref{ex:three_sym_real_intervals} with 
$\alpha = 0.2, 0.4, 0.49$ (from top to bottom).
Left: $E$ (black lines) and grid.  Right: $\partial L$ (black curves), $a_1, 
a_2, a_3$ (black dots) and the image under $\Phi$ of the grid.}
\label{fig:three_sym_real_intervals}
\end{figure}

\section{Several Intervals}
\label{sect:algorithm}

In this section, we consider the case that $E = P^{-1}(\cc{-1, 1})$ consists of 
$\ell$ components,
each symmetric with respect to the real line.  This includes in particular
the important case when $E$ consists of $\ell$ real intervals.

For a given polynomial $P$ with critical points $z_1, \ldots, z_{\ell-1} \in \C \setminus E$,
we want to compute the centers $a_1 < \ldots < a_\ell$ of $L$ with the help 
of~\eqref{eqn:ell_relation_crit_pts} and~\eqref{eqn:ell_sum_nj_aj}
in Theorem~\ref{thm:ell_components_real_symmetry}.
In other words, we are looking for the polynomial
\begin{equation*}
Q(w) = 2 p_n (w - a_1)^{n_1} \cdot \ldots \cdot (w - a_\ell)^{n_\ell},
\end{equation*}
where $n_1, \ldots, n_\ell$ are known (these are the number of zeros of $P$ in 
$E_1, \ldots, E_\ell$, see Theorem~\ref{thm:known_properties}\,(ii)),
such that the critical values of $Q$ at the critical points $w_1, \ldots, 
w_{\ell-1} \in \C \setminus L$ are given;
see~\eqref{eqn:ell_relation_crit_pts}.  Additionally, $a_1, \ldots, a_\ell$
must satisfy~\eqref{eqn:ell_sum_nj_aj}.

In his 1961 paper~\cite{Kammerer1961}, Kammerer considered a similar problem and
gave an algorithm for computing the (real) polynomial of degree $n$ with
$n-1$ prescribed oscillating critical values (and additionally two interpolatory 
conditions).
Existence and uniqueness of such a polynomial has been proved 
in~\cite{Davis1957};
see also~\cite{BeardonCarneNg2002, Kristiansen1984, Kuhn1969, 
MycielskiPaszkowski1960, Thom1965}
for further results on polynomials with prescribed critical values.
In~\cite[Sect.~6]{PeherstorferSchiefermayr1999}, Peherstorfer and the first
author provided a
modification of Kammerer's algorithm for computing a polynomial
pre-image consisting of $\ell$ intervals.

In the following, we modify Kammerer's algorithm for the computation of $Q$, that is,
for the computation of $a_1, \ldots, a_\ell$.  For the sake of simplicity,
let us concentrate first on the case of $\ell$ real intervals.

Let the polynomial $P$ be as in~\eqref{eqn:P} and such that
\begin{equation}
E=P^{-1}([-1,1])=[b_1,b_2]\cup[b_3,b_4]\cup\ldots\cup[b_{2\ell-1},b_{2\ell}]
\end{equation}
with $b_1<b_2<\ldots<b_{2\ell}$. Let $n_j$ be the number of zeros of $P$ in 
$E_j \coloneq [b_{2j-1},b_{2j}]$, $j=1,\ldots,\ell$, and let 
$z_1, \ldots, z_{\ell-1}$ be the critical points of $P$ outside~$E$ 
which satisfy $b_{2j} < z_j < b_{2j+1}$, $j = 1, \ldots, \ell-1$; see 
Theorem~\ref{thm:mapping_properties_Phi}\,\ref{it:known_properties_crit_pts}.  
The following algorithm 
gives a procedure for numerically computing 
the critical points $w_1,\ldots,w_{\ell-1}$ of $Q$ 
and the centers $a_1, \ldots, a_{\ell}$ of $L$
satisfying~\eqref{eqn:ell_relation_crit_pts} and~\eqref{eqn:ell_sum_nj_aj}
in Theorem~\ref{thm:ell_components_real_symmetry}.

\begin{algorithm} ~ \\
Initial values:
\begin{equation*}
\begin{aligned}
a_j^{[0]}&=\tfrac{1}{2}(b_{2j-1}+b_{2j}), \qquad j=1,\ldots,\ell \\
w_j^{[0]}&=\tfrac{1}{2}(b_{2j}+b_{2j+1}), \qquad j=1,\ldots,\ell-1
\end{aligned}
\end{equation*}
FOR $k = 0, 1, 2, \ldots$ DO
\begin{enumerate}
\renewcommand{\labelenumi}{\arabic{enumi}.}
\item Using $a_1^{[k]}, \ldots, a_\ell^{[k]}$ as initial values, compute 
$a_1^{[k+1]},\ldots,a_{\ell}^{[k+1]}$ such that
\begin{align*}
Q^{[k+1]}(w_j^{[k]}) &= P(z_j)+\sqrt{P^2(z_j)-1}, \quad j = 1, \ldots, \ell-1, 
\\
\sum_{j=1}^{\ell}n_ja_j^{[k+1]} &= -\frac{p_{n-1}}{p_n},
\end{align*}
where
\[
Q^{[k+1]}(w) 
= 2p_n\bigl(w-a_1^{[k+1]}\bigr)^{n_1}\cdot\ldots\cdot\bigl(w-a_{\ell}^{[k+1]}
\bigr)^{n_{\ell}}.
\]
\item Compute the solutions $w_1^{[k+1]},\ldots,w_{\ell-1}^{[k+1]}$ 
of the equation
\begin{equation*}
\sum_{i=1}^\ell n_i \prod_{\substack{j=1 \\ j \neq i}}^\ell \bigl( 
w-a_j^{[k+1]} \bigr) = 0.
\end{equation*}
\end{enumerate}
ENDFOR
\end{algorithm}

\begin{remark}
The algorithm also works for more general sets $E = \cup_{j=1}^n E_j$ 
with $E_j^* = E_j$ for $j = 1, \ldots, \ell$.  In this case, let us define $b_1, 
\ldots, b_{2\ell}$ by
\begin{equation}
E_j \cap \R = \cc{b_{2j-1}, b_{2j}}, \quad j = 1, \ldots, \ell.
\end{equation}
Note that $E_j \cap \R$ is indeed a point or an interval;
see~\cite[Lem.~A.2]{SchiefermayrSete2022}.
In particular, $b_1 \leq b_2 < b_3 \leq b_4 < \ldots < b_{2\ell-1} \leq b_{2 
\ell}$.  Then the algorithm works with the same initial values.
\end{remark}

In our MATLAB implementation of the above algorithm, we solve the nonlinear
system of equations for $a_1^{[k+1]}, \ldots$, $a_\ell^{[k+1]}$ in Step~1
using a Newton iteration, and compute the new values $w_1^{[k+1]},
\ldots, w_{\ell-1}^{[k+1]}$ in Step~2 using MATLAB's \verb|roots|
command.
We stop the iteration when
\begin{equation*}
\abs{a_j^{[k+1]} - a_j^{[k]}} < \operatorname{abstol} + \operatorname{reltol} \cdot \abs{a_j^{[k]}} \quad \text{for all } j = 1, \ldots, \ell,
\end{equation*}
where we have chosen $\operatorname{abstol} = \operatorname{reltol} = 10^{-13}$.
We performed the following experiments in MATLAB R2014b.

\begin{example} \label{ex:revisited}
We apply our algorithm to all previous examples.
Table~\ref{tab:algo} lists the number of iteration steps in our algorithm until convergence.
For those sets where $a_1, \ldots, a_\ell$ are known explicitly,
also the maximal error $\max_{j=1, \ldots, \ell} \abs{a_j^{[k]} - a_j}$
in the final step is reported.
We observe that the algorithm converges in very few iteration steps
in all examples.
Moreover, in the examples where $a_j$ are known explicitly, the
algorithm terminates with an error close to machine precision,
which suggests that the algorithm is highly accurate.
In Example~\ref{ex:two_sym_real_intervals}, the initial guess is already the exact solution and the iteration stops in its first step.
Figure~\ref{fig:convergence} shows the error curves $\abs{a_j^{[k]} -
a_j}$ for Example~\ref{ex:P3_two_intervals} and
Example~\ref{ex:three_sym_real_intervals}.

In Example~\ref{ex:three_intervals}, the exact values of $a_1$, $a_2$, $a_3$ 
are not known.  For a comparison, we first map $\comp{E}$ onto a
region bounded by Jordan curves (through the successive application of inverse Joukowski maps to the intervals) and then apply the numerical
boundary integral equation (BIE) method from~\cite{NasserLiesenSete2016}, which computes the centers
$a_1, \ldots, a_\ell$, the exponents $m_1, \ldots, m_\ell$, the capacity
$\capacity(E)$ and the conformal map $\Phi$.
The difference between the values $a_j$ computed by the two methods
is close to machine precision, suggesting that both methods are very accurate.
\end{example}

\begin{table}[t]
{\centering
\begin{tabular}{ccc}
\toprule
Example & iter. steps & max. error $\abs{a_j^{[k]} - a_j}$ \\
\midrule
\ref{ex:P3_two_intervals} & $4$ & $3.6637 \cdot 10^{-15}$ \\
\ref{ex:two_sym_real_intervals} & $0$ & $0$ \\
\ref{ex:intersecting_arcs} & $1$ & $0$ \\
\ref{ex:three_intervals} & $4$ & --- \\
\ref{ex:three_sym_real_intervals} & $4$ & $1.1102 \cdot 10^{-16}$ \\
\bottomrule
\end{tabular}

}
\caption{Number of iteration steps until convergence and final maximal error 
between the computed and exact values of $a_1, \ldots, a_\ell$, where 
available; see Example~\ref{ex:revisited}.}
\label{tab:algo}
\end{table}

\begin{figure}
{\centering
\includegraphics[width=0.48\linewidth]{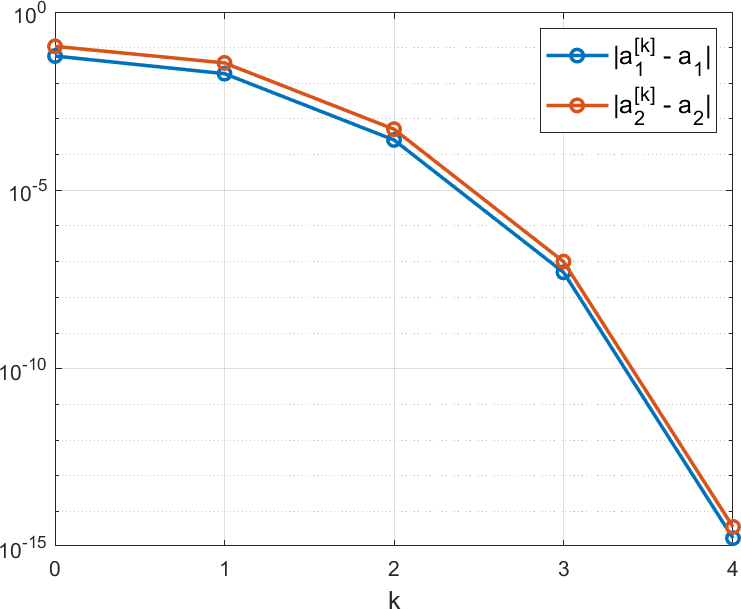}
\includegraphics[width=0.48\linewidth]{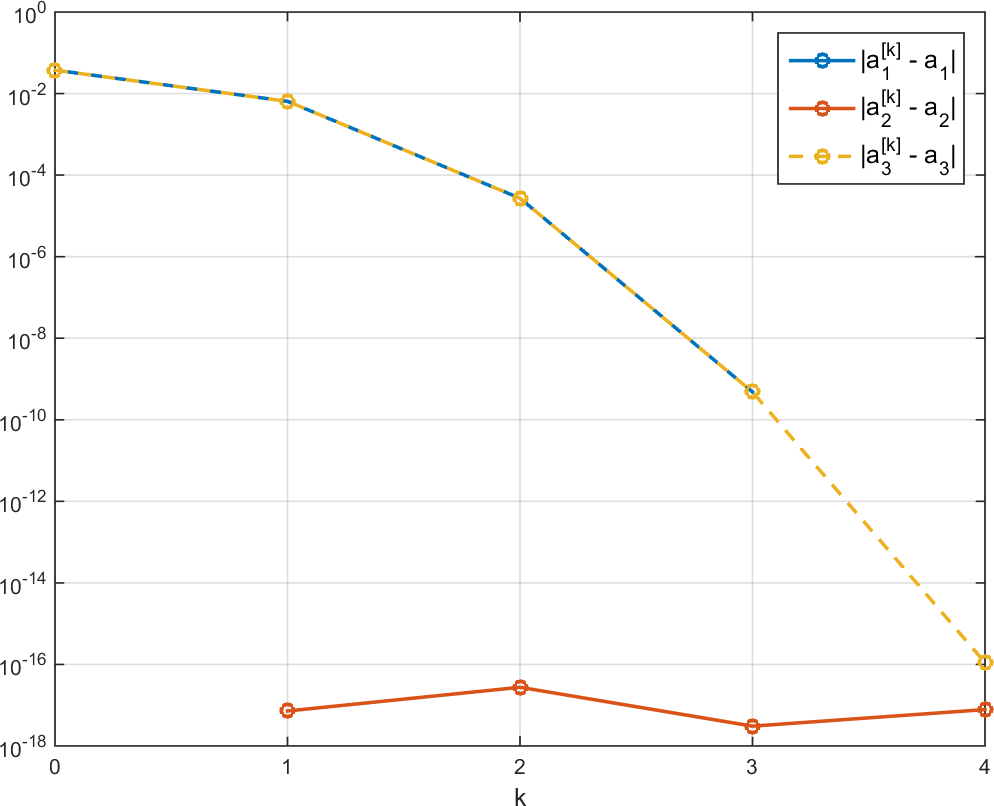}

}
\caption{Error curves $\abs{a_j^{[k]} - a_j}$ for the sets $E$ from 
Example~\ref{ex:P3_two_intervals} (left)
and~\ref{ex:three_sym_real_intervals} (right).
Missing dots mean that the error is exactly zero.}
\label{fig:convergence}
\end{figure}

\begin{example} \label{ex:deg7_ell4}
The polynomial $P$ shown in Figure~\ref{fig:E_ell_intervals} is
\begin{align*}
P(z) &= -75.60176146228515 z^7 -6.112631353464664 z^6 + 130.38101983617594 z^5 
\\
&\phantom{=} + 7.91625847587283 z^4 - 63.44786532361087 z^3 
- 1.793124210064775 z^2 \\
&\phantom{=} + 7.668606949720056 z -0.010502912343433701.
\end{align*}
The set $E = P^{-1}(\cc{-1, 1})$ consists of $\ell = 4$ intervals while $n = 
\deg(P) = 7$.
The numbers of zeros of $P$ in the intervals are $n_1 = 2$, $n_2 = 1$, $n_3 = 
3$, $n_4 = 1$, which can also be seen from Figure~\ref{fig:E_ell_intervals}.
Our algorithm converges after $5$ iteration steps to the values
\begin{align*}
a_1 &= -0.807906463544657 & a_3 &= 0.341284084426686 \\
a_2 &= -0.367217238438923 & a_4 &= 0.878324884021925.
\end{align*}
For comparison, we compute $a_1, \ldots, a_4$ by adapting the BIE method 
from~\cite{NasserLiesenSete2016} as described in Example~\ref{ex:revisited}.
The difference of the computed values is of order $10^{-15}$, suggesting that 
both methods are very accurate.

With the obtained values $a_1, \ldots, a_4$, the conformal map $w = \Phi(z)$ is 
evaluated as follows.
We solve the polynomial equation $Q(w) = P(z) + \sqrt{P(z)^2 - 1}$, see 
Theorem~\ref{thm:known_properties}\,\ref{it:relation_between_maps}, and
determine the correct value of $w$ using the mapping
properties of $\Phi$ given in Theorem~\ref{thm:mapping_properties_Phi}.
Figure~\ref{fig:deg7_ell4} shows the sets $E$ and $L$ and the image of a grid 
under the conformal map $\Phi$.
\end{example}

\begin{figure}
{\centering
\includegraphics[width=0.48\linewidth]{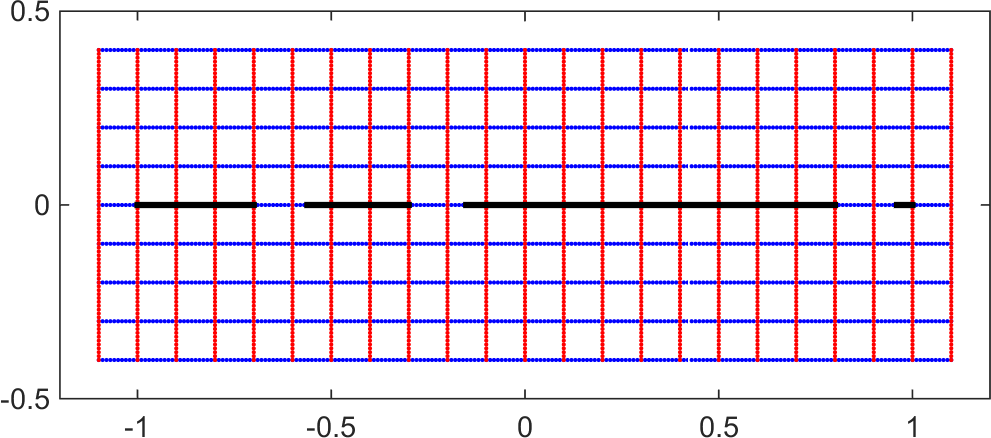}
\includegraphics[width=0.48\linewidth]{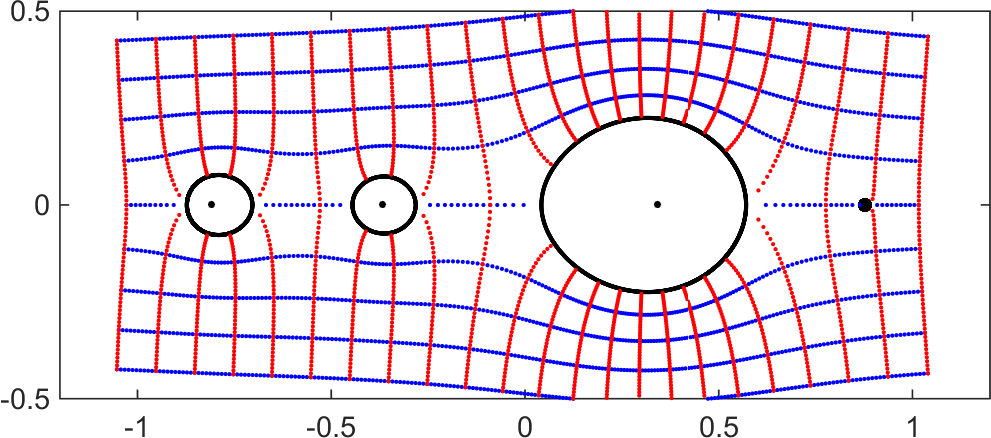}

}
\caption{Left: Polynomial pre-image $E = P^{-1}(\cc{-1, 1})$ with $\ell
= 4$ intervals (black) and $\deg(P) = 7$ from Example~\ref{ex:deg7_ell4} 
and a grid.
Right: $\partial L$ (black curves), $a_1, \ldots, a_4$ (black dots),
and the image of the grid under the conformal map $\Phi$.}
\label{fig:deg7_ell4}
\end{figure}

\begin{example} \label{ex:five_intervals}
Consider the polynomial
\begin{equation*}
P(z) = 26 z^5 + 5 z^4 - 32 z^3 - 5 z^2 + 7 z + \frac{1}{2},
\end{equation*}
for which $E = P^{-1}(\cc{-1, 1})$ consists of $\ell = 5$ intervals;
see Figure~\ref{fig:five_intervals}.  Here, $n_1 = \ldots = n_5 = 1$.
The centers $a_1, \ldots, a_5$ of $L$ are not known explicitly.
Our algorithm converges in $4$ steps to the values
\begin{align*}
a_1 &= -0.957893296657925, & a_4 &= 0.464367835203743, \\
a_2 &= -0.570567929561560, & a_5 &= 0.951037765762270. \\
a_3 &= -0.079252067054220,
\end{align*}
For comparison, we compute $a_1, \ldots, a_5$ also
by adapting the BIE method from~\cite{NasserLiesenSete2016} as 
described in Example~\ref{ex:revisited}.
The difference of the computed results is of the order $10^{-14}$, i.e.,
both results agree up to almost machine precision.

The computed values $a_1, \ldots, a_5$ are used to compute numerically the 
conformal map~$\Phi$ as described in Example~\ref{ex:deg7_ell4}.
Figure~\ref{fig:five_intervals} illustrates the conformal map.
The difference of the values of $\Phi$ on the grid computed by this
method and by the BIE method from~\cite{NasserLiesenSete2016} is
of the order $10^{-14}$.  We have again a very good agreement
between both methods, which suggests that both methods are very
accurate.
\end{example}

\begin{figure}
{\centering
\includegraphics[width=0.48\linewidth]{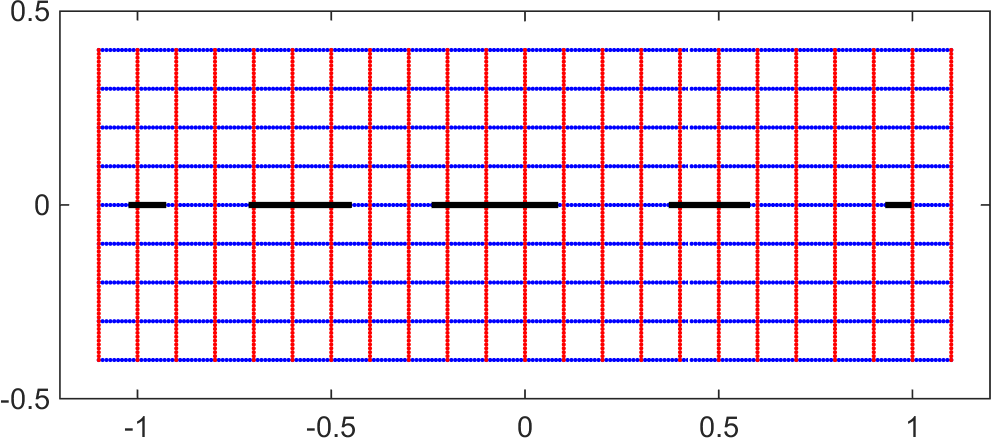}
\includegraphics[width=0.48\linewidth]{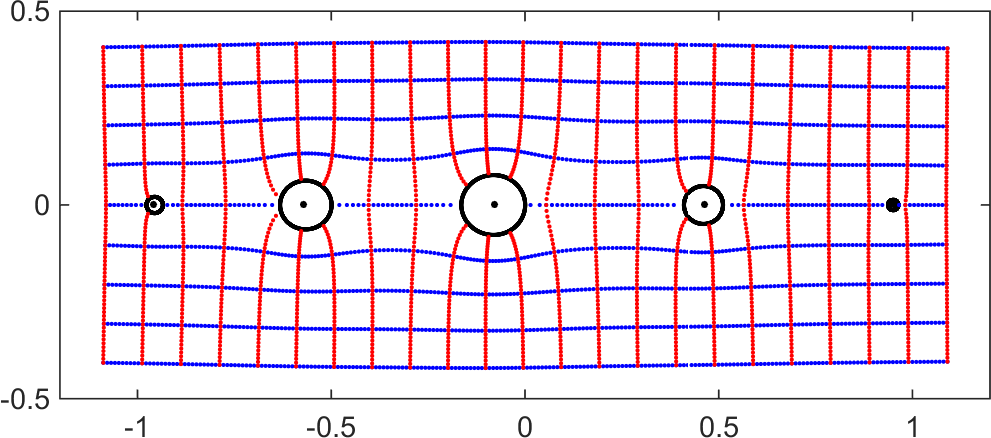}

}
\caption{Left: Polynomial pre-image $E = P^{-1}(\cc{-1, 1})$ with $\ell
= 5$ intervals (black) from Example~\ref{ex:five_intervals} and a grid.
Right: $\partial L$ (black curves), $a_1, \ldots, a_5$ (black dots),
and the image of the grid under the conformal map $\Phi$.}
\label{fig:five_intervals}
\end{figure}

\begin{example}\label{ex:ten_intervals}
Let $P(z) = \alpha T_{10}(z)$, where $\alpha = 1.05$ and $T_{10}$ is the
Chebyshev polynomial of the first kind of degree $n = 10$.
Then $E = P^{-1}(\cc{-1, 1})$ consists
of $\ell = 10$ intervals; see Figure~\ref{fig:ten_intervals}.
Similarly to Example~\ref{ex:five_intervals}, we compute $a_1, \ldots,
a_{10}$ with our algorithm, which converged in $5$ iteration steps.
In addition, we compute the conformal map $\Phi$,
illustrated in Figure~\ref{fig:ten_intervals}.
We repeated the same numerical experiment with $P(z) = \alpha T_{20}(z)$ with $\alpha = 1.05$, for which
$E = P^{-1}(\cc{-1, 1})$ consists of $20$ intervals.  Our algorithm for
computing $a_1, \ldots, a_{20}$ converged in $6$ steps.
\end{example}

\begin{figure}
{\centering
\includegraphics[width=0.48\linewidth]{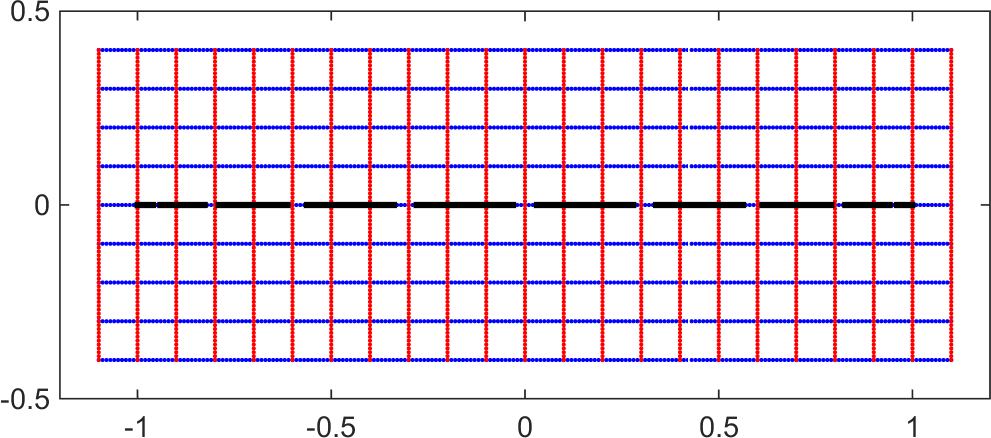}
\includegraphics[width=0.48\linewidth]{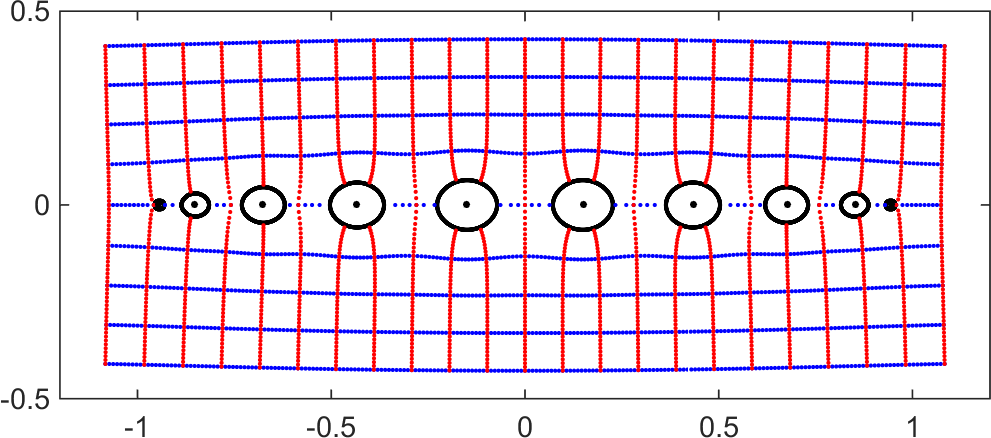}

}
\caption{Left: Polynomial pre-image $E = P^{-1}(\cc{-1, 1})$ with $\ell
= 10$ intervals (black) from Example~\ref{ex:ten_intervals} and a grid.
Right: $\partial L$ (black curves), $a_1, \ldots, a_{10}$ (black dots),
and the image of the grid under the conformal map $\Phi$.}
\label{fig:ten_intervals}
\end{figure}

\appendix
\section{Appendix}
\label{sect:appendix}

Given a set $E = \cup_{j=1}^\ell E_j$ with certain symmetries, we characterize 
the 
possible relations between the components $E_1, \ldots, E_\ell$.
In this section, unlike in the previous parts of the paper, the components 
$E_j$ are allowed to shrink to points, and we require that the $E_j$ are 
nonempty simply connected compact sets.  Note that if $E_j$ has at least two 
points, then it is an infinite compact set.

\begin{theorem} \label{thm:symmetric_E}
Let $E = \cup_{j=1}^\ell E_j$ with $E = -E$ and disjoint, simply connected 
nonempty compact sets $E_1, \ldots, E_\ell$.
\begin{enumerate}
\item \label{it:sym_E_contains_origin}
If $\ell = 1$ then $0 \in E$.
\item If $\ell = 2$ then $E_1 = -E_2$ and $0 \notin E$.
\item If $\ell = 3$ then there exists $j_0, j_1, j_2$ with $\{ j_0, j_1, j_2 \} 
= \{ 1, 2, 3 \}$ with $0 \in E_{j_0} = -E_{j_0}$ and $E_{j_1} = -E_{j_2}$.
\item If $\ell = 2k$ is even then $\{ 1, \ldots, \ell \}$ is partitioned into 
$k$ pairs $(j_1, j_2)$ with $E_{j_1} = -E_{j_2}$, and $0 \notin E$.
\item If $\ell = 2k+1$ is odd then $\{ 1, \ldots, \ell \}$ is partitioned into 
$\{ j_0 \}$ with $0 \in E_{j_0} = -E_{j_0}$ and $k$ pairs $(j_1, j_2)$ with 
$E_{j_1} = -E_{j_2}$.
\end{enumerate}
\end{theorem}

\begin{proof}
First, let $\ell = 1$, i.e., let $E$ be a simply connected nonempty compact set 
with $E = -E$.  If $E$ consists of a single point then $E = -E$ implies $E = \{ 
0 \}$.  Otherwise, $E$ is an infinite compact set 
(since $E$ is simply connected with at least two distinct points), and we can 
apply Theorem~\ref{thm:walsh_map} to $E$.  The symmetry $E = -E$ implies that 
$L = \{ w \in \C : \abs{w} \leq \capacity(E) \}$ since $a_1 = 0$ by 
Theorem~\ref{thm:aj_symm_origin}, and that $\Phi$ is odd.
The assumption $0 \in \comp{E}$ leads to the contradiction $0 = \Phi(0) \in 
\comp{L}$.  This shows that $0 \in E$ and concludes the proof 
of~\ref{it:sym_E_contains_origin}.

Let $\ell \geq 2$ and $E = \cup_{j=1}^\ell E_j$ with $E = -E$.  The function 
$f(z) = -z$ satisfies $f(E) = E$.
For $j \in \{ 1, 2, \ldots, \ell \}$, the set $f(E_j)$ is connected and, since 
$f(E_j) \subseteq E$, there exists $k_j \in \{ 1, \ldots, \ell \}$ with $f(E_j) 
\subseteq E_{k_j}$.
Define $\sigma : \{ 1, \ldots, \ell \} \to \{ 1, \ldots, \ell \}$, $j \mapsto 
k_j$.  We show that $\sigma$ is a permutation.
Suppose it is not, then there exists $k_0 \in \{ 1, \ldots, \ell \}$ such that 
$f(E_j) \not\subseteq E_{k_0}$ for all $j$, which implies
$E = f(E) = \cup_{j=1}^\ell f(E_j) \subseteq \cup_{j=1, j \neq k_0}^\ell E_j 
\subsetneq E$, a contradiction.
Thus $\sigma$ is a permutation.  Since $f(E) = E$, we also have $f(E_j) = 
E_{k_j}$ for all $j \in \{ 1, \ldots, \ell \}$.

We write the permutation $\sigma$ as a product of disjoint cycles.
Since $f$ is an involution, i.e., $f(f(z)) = z$, we have $\sigma^2 = \id$, 
which shows that the cycles have length $1$ or $2$ (transpositions).
A cycle of length $1$, say $(j)$, corresponds to a set with $-E_j = f(E_j) = 
E_j$, which satisfies $0 \in E_j$ by~\ref{it:sym_E_contains_origin}.  Since 
the sets $E_1, \ldots, E_\ell$ are disjoint, there can be no other cycle of 
length $1$.
Thus, $\sigma$ is the product of (disjoint) transpositions and possibly one 
cycle of length $1$.

The remaining assertions now follow easily.
If $\ell = 2$ then $\sigma = (1, 2)$ and $f(E_1) = E_2$, i.e., $E_1 = -E_2$.
More generally, if $\ell = 2k$ is even, $\sigma$ is a product of $k$ (disjoint) 
transpositions and there are $k$ pairs $(j_1, j_2)$ which partition $\{ 1, 
\ldots, 2k \}$, such that $E_{j_1} = -E_{j_2}$.
In particular, $0 \notin E$, since otherwise $0 \in E_{j_1} = -E_{j_2}$ implies 
$0 \in E_{j_1} \cap E_{j_2}$, a contradiction.
If $\ell = 3$ then $\sigma = (j_0) (j_1, j_2)$ with $\{ j_0, j_1, j_2 \} = \{ 
1, 2, 3 \}$, that is $E_{j_1} = -E_{j_2}$ and $E_{j_0} = -E_{j_0}$ with $0 \in 
E_{j_0}$.
More generally, if $\ell = 2k+1$ is odd, $\sigma$ is a product of one cycle 
of length $1$ and $k$ (disjoint) transpositions, i.e., there is one $j_0 \in \{ 
1, \ldots, 2k+1 \}$ with $0 \in E_{j_0} = - E_{j_0}$ and $k$ pairs $E_{j_1} = 
-E_{j_2}$.
\end{proof}

The method of proof of Theorem~\ref{thm:symmetric_E} yields an analogous
statement for sets that are symmetric with respect to the real line.

\begin{corollary} \label{cor:E_symm_wrt_R}
Let $E = \cup_{j=1}^\ell E_j$ with $E^* = E$ and disjoint, simply connected 
nonempty compact sets $E_1, \ldots, E_\ell$.
Then $\{ 1, 2, \ldots, \ell \}$ is partitioned into sets $\{ j_0 
\}$ with one element, for which $E_{j_0}^* = E_{j_0}$, and sets $\{ j_1, j_2 
\}$ with two distinct elements, for which $E_{j_1}^* = E_{j_2}$.

In the special case $\ell = 2$, i.e., $E = E_1 \cup E_2$ with $E^* = E$, then 
either $E_1^* = E_1$ and $E_2^* = E_2$, or $E_1^* = E_2$.
\end{corollary}

\begin{proof}
Let $f(z) = \conj{z}$ and $f(E) = E^* = E$, for which $f(f(z)) = z$ for all 
$z \in \C$.  Proceeding as in the proof of Theorem~\ref{thm:symmetric_E}, we obtain that $f(E_j) = E_{\sigma(j)}$ with a permutation $\sigma$, which decomposes into cycles of length $1$ or $2$.
Hence the components of $E$ are either symmetric ($E_j^* = E_j$, corresponding 
to a cycle $(j)$) or come in pairs $E_{j_1}^* = E_{j_2}$ (corresponding to a 
cycle $(j_1, j_2)$ of length $2$).
\end{proof}

\begin{remark} \label{rem:symmetric_E}
The method of proof of Theorem~\ref{thm:symmetric_E} can be generalized 
to $E = \cup_{j=1}^\ell E_j$ with disjoint, simply connected nonempty compact 
sets $E_1, \ldots, E_\ell$, and a continuous function $f : E \to E$ with $f(E) 
= E$ and $f^k = f \circ \ldots \circ f = \id$ for some integer $k \geq 2$ 
(where $k$ is minimal with this property).
Then $\sigma$ defined as in the proof of Theorem~\ref{thm:symmetric_E}
is a permutation that satisfies $\sigma^k = \id$.  In 
particular, $\sigma$ can be written as a product of disjoint cycles, where the 
length of each cycle divides $k$.
Indeed, $\sigma^k = \id$ implies $c^k = \id$ for every cycle $c$ of $\sigma$,
hence a cycle $c$ has length $l(c) \leq k$.  Write $k 
= q l(c) + r$ with $0 \leq r < l(c)$.  Then $\id = c^k = c^r$, which 
implies $r = 0$ since $r < l(c)$ and $l(c)$ is the smallest positive
exponent with $c^{l(c)} = \id$.  Hence, $l(c)$ divides $k$.
\end{remark}

\bibliographystyle{siam}
\bibliography{walshmap.bib}

\end{document}

%% file: def_mathfonts.tex
\newcommand{\C}{\mathbb{C}}
\newcommand{\R}{\mathbb{R}}

\newcommand{\bD}{\mathbb{D}}

\newcommand{\cO}{\mathcal{O}}

\newcommand{\cR}{\mathcal{R}}



%% file: def.tex
\newcommand{\coloneq}{\mathrel{\mathop:}=}
\newcommand{\eqcolon}{=\mathrel{\mathop:}}

\newcommand{\conj}[1]{\overline{#1}}
\newcommand{\comp}[1]{#1^{\operatorname{c}}}

\newcommand{\dd}{{\operatorname{d}}}
\newcommand{\ee}{{\operatorname{e}}}
\newcommand{\ii}{{\operatorname{i}}}

\DeclarePairedDelimiter{\abs}{\lvert}{\rvert}

\DeclarePairedDelimiter{\cc}{[}{]}

\DeclarePairedDelimiter{\oo}{]}{[}

\DeclareMathOperator{\re}{Re}
\DeclareMathOperator{\im}{Im}

\DeclareMathOperator{\capacity}{cap}
\DeclareMathOperator{\id}{id}

\DeclareMathOperator{\wind}{wind}

\DeclareMathOperator{\Ri}{\mathcal{R}}